\documentclass{amsart}
\usepackage{amsmath}
\usepackage{amsfonts}
\usepackage{amssymb}
\usepackage{amsthm}
\usepackage{comment}
\usepackage{graphicx}

\usepackage{tikz}
\usetikzlibrary{patterns}

\newtheorem{thm}{Theorem}[section]
\newtheorem{cor}[thm]{Corollary}
\newtheorem{lem}[thm]{Lemma}

\theoremstyle{definition}
\newtheorem{defn}[thm]{Definition}

\theoremstyle{remark}
\newtheorem{rem}[thm]{Remark}

\newcommand{\norm}[1]{\left\Vert#1\right\Vert}

\newcommand{\abs}[1]{\left\vert#1\right\vert}
\newcommand{\set}[1]{\left\{#1\right\}}
\newcommand{\R}{\mathbb R}
\newcommand{\C}{\mathbb C}
\newcommand{\Z}{\mathbb Z}

\newcommand{\To}{\rightarrow}

\DeclareMathOperator{\HM}{HM}
\DeclareMathOperator{\ess}{ess}

\def\normQR#1#2#3{\norm{#1}_{L^{#2}_tL^{#3}_x}}

\def\normP#1#2{\norm{#1}_{L^{#2}}}

\def\Int{\int_{-\infty}^{\infty} }
\def\inT{\int_{0}^{t}}

\def\intI{\int_{-\infty}^{t}}

\def\rp#1{\frac 1 {#1}}

\def\Sch{Schr\"odinger }

\newcommand{\ktoper}[1]{\partial_t #1(t,x,v) + v\cdot\nabla_x #1(t,x,v)}
\def\coor{(t,x,v)}
\def\normA#1#2{\norm{#1}_{#2}}

\def\Ol{\mathcal O_\lambda}

\def\normLx#1#2{\norm{#1}_{L^{#2}_x}}
\def\spaceLx#1{{L^{#1}_x}}
\def\normLv#1#2{\norm{#1}_{L^{#2}_v}}

\def\trip{(q,r,p)}
\def\tript{(\tilde q, \tilde r, \tilde p)}
\def\triptrip{(1/q, 1/r, 1/p, 1/\tilde q, 1/ \tilde r, 1/ \tilde p)}
\def\Triptrip{(1/Q, 1/R, 1/P, 1/\tilde Q, 1/ \tilde R, 1/ \tilde P)}
\def\normQRP#1#2#3#4{\norm{#1}_{L^{#2}_tL^{#3}_xL^{#4}_v}}
\def\normQRPloc#1#2#3#4#5{\norm{#1}_{L^{#2}_t(#5; L^{#3}_xL^{#4}_v)}}
\def\normA#1#2{\norm{#1}_{L^{#2}_{x,v}}}
\def\normBC#1#2#3{\norm{#1}_{L^{#2}_{x}L^{#3}_{v}}}
\def\normQRPloc#1#2#3#4#5{\norm{#1}_{L^{#2}_t(#5;L^{#3}_xL^{#4}_v)}}

\def\normQRPf#1#2#3#4{\norm{#1}_{L^{#2}_tL^{#3}_xL^{#4}_v(V)}}
\def\normQRPfd#1#2#3#4{\norm{#1}_{L^{\tilde #2'}_tL^{\tilde #3'}_xL^{\tilde #4'}_v(V)}}
\def\spaceLqrpftd{L^{\tilde q'}_tL^{\tilde r'}_xL^{\tilde p'}_v(V)}
\def\spaceLqrpf{L^{q}_tL^{r}_xL^{p}_v(V)}

\def\Ol{\mathcal{O}_\lambda}

\def\normInfty#1{\norm{#1}_{\infty}}
\def\normQRP#1#2#3#4{\norm{#1}_{L^{#2}_tL^{#3}_xL^{#4}_v}}
\def\normQRPD#1#2#3#4{\norm{#1}_{L^{\tilde #2'}_tL^{\tilde #3'}_xL^{\tilde #4'}_v}}
\def\normQRPtd#1#2#3#4{\norm{#1}_{L^{\tilde #2'}_tL^{\tilde #3'}_xL^{\tilde #4'}_v}}

\def\normQRPloc#1#2#3#4#5{\norm{#1}_{L^{#2}_t(#5; L^{#3}_xL^{#4}_v)}}
\def\normQRPDloc#1#2#3#4#5{\norm{#1}_{L^{\tilde #2'}_t(#5; L^{\tilde #3'}_xL^{\tilde #4'}_v)}}

\def\normLa#1#2{\norm{#1}_{L^{#2}_{x,v}}}
\def\normLbc#1#2#3{\norm{#1}_{L^{#2}_{x}L^{#3}_v}}

\def\spaceLtwo{L^2_{x,v}}
\def\spaceLa#1{L^{#1}_{x,v}}
\def\spaceLbc{L^{b}_{x} L^{c}_{v}}

\def\spaceLqrp{L^q_tL^r_xL^p_v}
\def\spaceLqrpd{L^{q'}_tL^{r'}_xL^{p'}_v}
\def\spaceLqrptd{L^{\tilde q'}_tL^{\tilde r'}_xL^{\tilde p'}_v}
\def\SpaceLqrp#1#2#3{L^{#1}_tL^{#2}_xL^{#3}_v}
\def\SpaceLqrploc#1#2#3#4{L^{#1}_t(#4; L^{#2}_xL^{#3}_v)}
\def\SpaceLqrpd#1#2#3{L^{#1'}_tL^{#2'}_xL^{#3'}_v}
\def\SpaceLqrptd#1#2#3{L^{\tilde #1'}_tL^{\tilde #2'}_xL^{\tilde #3'}_v}
\def\SpaceLqrptdloc#1#2#3#4{L^{\tilde #1'}_t(#4, L^{\tilde #2'}_xL^{\tilde #3'}_v)}
\def\SpaceLqrptdMod#1#2#3{L^{#1}_tL^{\tilde #2'}_xL^{\tilde #3'}_v}
\def\SpaceLqrpdMod#1#2#3#4{L^{#1', #2}_tL^{#3'}_xL^{#4'}_v}
\def\spaceQRPlocd#1#2#3{L^{\tilde #1'}_{t}([0,T];L^{\tilde #2'}_xL^{\tilde #3'}_v)}

\def\normXVV#1#2#3#4{\norm{#1}_{L^{#2}_xL^{#3}_vL^{#4}_{v'}}}

\def\normP#1#2{\norm{#1}_{L^{#2}_x}}

\def\qtet{(q,r,p,b,c)}

\def\normLL#1{\norm{#1}_{L^{q}(\lambda I; L^{r}_xL^{p}_v)}}
\def\normRL#1{\norm{#1}_{L^{\tilde q'}(\lambda J; L^{\tilde r'}_xL^{\tilde p'}_v)}}
\def\bbeta{\frac 1 q + \frac 1 {\tilde{q}} - n\left(1 - \frac 1 r - \rp {\tilde{r}}\right)}
\def\betta{{n\left(\frac 1 p - \frac 1 r\right)}}
\def\qconds{\frac 1 {\tilde{r}} - \frac 1 {\tilde{p}}  - \frac 1 {r} + \frac 1 {p} \leq \frac 2 {n q}}
\def\qcondts{\frac 1 r - \frac 1 p  - \frac 1 {\tilde r} + \frac 1 {\tilde p} \leq \frac 2 {n \tilde{q}}}

\def\scondAB{\rp R + \rp P + \rp {\tilde R} + \rp {\tilde P} = 2}

\def\triptrip{(1/q, 1/r, 1/p, 1/\tilde q, 1/ \tilde r, 1/ \tilde p)}
\def\Triptrip{(1/Q, 1/R, 1/P, 1/\tilde Q, 1/ \tilde R, 1/ \tilde P)}

\def\Str{Strichartz }

\begin{document}

\title{Strichartz Estimates for the Kinetic Transport Equation}
\author{Evgeni Y Ovcharov}%
\address{Angewandte Mathematik und Bioquant, Universit\"at Heidelberg, INF 267, Heidelberg 69120, Germany}
\email{evgeni.ovcharov@bioquant.uni-heidelberg.de}

\thanks{Much of this work was done while at the University of Edinburgh. The author would like to thank Damiano Foschi for the several suggestions he made about this work.}

\date{\today}

\begin{abstract}
In this paper we prove Strichartz estimates for the kinetic transport equation and make a detailed investigation on their range of validity. In one spatial dimension we find essentially all possible estimates, while in higher dimensions some endpoint and inhomogeneous estimates remain open. The remaining estimates are analogous to the remaining open inhomogeneous Strichartz estimates in other contexts. The Strichartz estimates that we derive extend the previous work by Castella and Perthame \cite{CP} (1996) and  Keel and Tao \cite{KT} (1998) in the context of the kinetic transport equation, and the techniques of Foschi \cite{F} to the current setting.
\end{abstract}

\maketitle

\section{Introduction}
The purpose of this paper is to study the range of validity of the Strichartz estimates for the kinetic transport (KT) equation
\begin{align}
 &\ktoper{u} = F\coor, \quad \coor \in (0,\infty)\times\R^n\times \R^n, \label{eq: KT 1} \\
 &u(0,x,v) = f(x,v).\label{eq: KT 2}
\end{align}
The solution $u$ to \eqref{eq: KT 1}, \eqref{eq: KT 2} has the form $u=U(t)f + W(t)F$, where
\begin{align*}
 U(t)f = f(x-tv, v), \quad W(t)F = \intI U(t-s) F(s) ds,
\end{align*}
and supp $F \subseteq (0, \infty)$. We want to study estimates of the form
\begin{equation*}
 \normQRP{u}{q}{r}{p} \lesssim \normA{f}{a} + \normQRP{F}{\tilde q'}{\tilde r'}{\tilde p'},
\end{equation*}
where $\spaceLqrp$ stands for the mixed Lebesgue space $L^q((0,\infty); L^r(\R^n; L^p(\R^n)))$, and $\spaceLa{a}$ stands for $L^a(\R^{2n})$. In the sequel we shall study separately the full range of \Str estimates for each operator $U(t)$ and $W(t)$. The \Str estimates for the KT equation of the form
\begin{equation}\label{eq: hom Strich KT}
 \normQRP{U(t)f}{q}{r}{p} \lesssim \normA{f}{a}
\end{equation}
are called homogeneous, while the estimates
\begin{equation}\label{eq: inhom Strich KT}
 \normQRP{W(t)F}{q}{r}{p} \lesssim \normQRP{F}{\tilde q'}{\tilde r'}{\tilde p'}
\end{equation}
are called inhomogeneous. As it is typically done, we shall prove estimates \eqref{eq: hom Strich KT} and \eqref{eq: inhom Strich KT} under the slightly more general assumptions $t \in \R$ and supp $F \subseteq (-\infty, \infty)$. For the sake of simplicity we shall again use the same notation $\spaceLqrp$ for the space $L^q(\R; L^r(\R^n; L^p(\R^n)))$.

\Str estimates for the KT equation appeared first in the note of Castella and Perthame \cite{CP} (1996) where a range of homogeneous estimates and some special inhomogeneous estimates are presented. In the seminal paper of Keel and Tao \cite{KT} (1998) the authors dedicate a small paragraph to the KT equation where they extend the homogeneous estimates proved in \cite{CP}. However, the endpoint homogeneous estimate proves too difficult to be resolved by the methods presented in \cite{KT}, which initiates an ongoing mathematical investigation. The first partial answer in that direction is given by Guo and Peng \cite{GP} (2007) who provide counterexamples in one spatial dimension confirming the (expected) failure of the endpoint estimate
\begin{align} \label{est: endpoint KT}
 \normQRP{U(t)f}{2}{\infty}{1} \lesssim \normA{f}{2}
\end{align}
there.

Presently, we extend the work of the previously mentioned authors by making a detailed analysis of the range of validity of the \Str estimates for the KT equation. The new estimates that we prove concern mostly the inhomogeneous operator $W(t)$ but we also prove new estimates for $U(t)$ of the more general form
\begin{equation}\label{eq: hom Strich mix}
 \normQRP{U(t)f}{q}{r}{p} \lesssim \normBC{f}{b}{c}.
\end{equation}
In fact we prove that the latter estimates are equivalent to some special inhomogeneous estimates \eqref{eq: inhom Strich KT} which explains why the investigation of the inhomogeneous estimates is to us of primary interest.

As a motivation for studying \Str estimates for the KT equation we can point out the fact that they have been a very fruitful tool in the context of the wave and the \Sch equations in the analysis of the nonlinear Cauchy problem. Application of such type appeared in Bournaveas et al. \cite{B} (2008) where the authors prove the existence of some weak solutions to a nonlinear kinetic system modeling chemotaxis.

The paper is organized as follows. In the next section we present the main results of the paper which we hope are given in a form convenient for referencing. In the section immediately after it we make some additional
introductory remarks that will be useful to those who wish to read on with our proofs. They follow in the sections after it.

\section{Strichartz estimates for the KT equation} \label{sect: main}

In order to formulate our results we first need several definitions.

\begin{defn} \label{dfn: admiss 2}
We say that the exponent triplet $\trip$, $1 \leq p,q, r \leq \infty$, is \textit{KT-admissible} if
\begin{gather}
 \rp q = \frac n 2 \left(\rp p - \rp r \right), \quad a \stackrel{def}{=} \HM(p,r),\label{eq: KT admiss dfn 1}\\
 1\leq a \leq \infty, \quad p^*(a) \leq p \leq a, \quad a \leq r \leq r^*(a), \label{eq: KT admiss dfn 2}
\end{gather}
except in the case $n=1$, $\trip=(a, \infty, a/2)$.
\end{defn}
In the above definition by $\HM(p,r)$ we have denoted the harmonic mean of $p$ and $r$, i.e. $a=\HM(p,r)$ whenever
\[
  \rp a = \rp 2 \left( \rp r + \rp p \right).
\]
For convenience we give explicitly the exact lower boundary $p^*$ to $p$ and the exact upper boundary $r^*$ to $r$ which are given in
\begin{align}\label{cond: r* KT}
\begin{cases}
 p^*(a) = \frac {na} {n+1}, \quad r^*(a) =  \frac {na} {n-1},   \quad \text{if} \quad \frac {n+1} n \leq a \leq \infty,\\
 p^*(a) = 1,   \qquad \;     r^*(a) =  \frac {a} {2-a},         \quad \text{if} \qquad  1\leq a\leq \frac {n+1} n.
\end{cases}
\end{align}
Note that the second line in \eqref{cond: r* KT} is placed to restrict the Lebesgue exponents $p$ and $r$ in the range $[1,\infty]$.

We have used above the convention that $1/0=\infty$, and thus, for example, for $n=1$ $r^*(a)=\infty$. Furthermore, throughout this text we shall always use the convention $1/\infty=0$ and $1/0=\infty$ in the context of Lebesgue exponents. Triplets of the form $\trip=(a,r^*(a), p^*(a))$, for $(n+1)/n\leq a<\infty$, will be called endpoint. The H\"older conjugate exponent will be denoted by $'$ e.g. $1/r+1/r'=1$. Conditions \eqref{eq: KT admiss dfn 2}, \eqref{cond: r* KT} are equivalent to $a \leq q$, and $p\leq r$, and $1\leq a$, $p$, condition \eqref{eq: KT admiss dfn 1} is equivalent to
\[
 \rp q + \frac n r = \frac n a, \quad \HM(p,r)=a.
\]
Note that although the latter redaction of condition \eqref{eq: KT admiss dfn 1} resembles more closely the admissability conditions for the wave and the \Sch equations, the former version is more natural in the present context in view of the fact that in the inhomogeneous estimates the Lebesgue exponent $a$ does not appear explicitly.

To describe the range of the inhomogeneous estimates we shall need the next two definitions. Following Foschi \cite{F}, we give the following

\begin{defn} \label{dfn: accept}
We say that the exponent triplet $\trip$ is \textit{KT-acceptable} if
\begin{align}\label{cond: KT accept}
\rp q < n \left( \rp p - \rp r \right), \quad 1\leq q\leq\infty, \quad 1 \leq p < r \leq \infty,
\end{align}
or if $q=\infty$, $1 \leq p=r \leq \infty$.
\end{defn}

Note that a KT-acceptable triplet is always KT-admissible. To further describe the range of validity of the inhomogeneous estimates we give the following

\begin{defn} \label{dfn: joint accept}
We say that the two KT-acceptable exponent triplets $\trip$ and $\tript$ are {\it jointly KT-acceptable} if
\begin{align}
\rp q + \rp {\tilde q}& = n\left(1 - \rp r - \rp {\tilde r}\right), \quad \rp q + \rp {\tilde q} \leq 1, \label{cond: joint KT accept 1}\\
&\HM(p,r)=\HM(\tilde p', \tilde r'),  \label{cond: joint KT accept 2} 
\end{align}
and if the exponents satisfy further the additional restrictions
\begin{itemize}
 \item[(i)] for $r,\tilde r \neq \infty$
 \begin{equation}\label{cond: glob inhom KT}
        \frac {n -1}{p'} < \frac {n}{\tilde{r}}, \quad \frac {n-1}{\tilde{p}'} < \frac {n}{r},
\end{equation}
 \item[(ii)] if $r = \infty$ then the point $\triptrip \in \Sigma_1 \cup B$,
\begin{align} \label{cond: sigma1}
 \begin{split}
 \Sigma_1 &=\set{\left(\mu,0,\kappa,\nu,1 - \kappa,1\right)\;:\;0 < \mu, \nu < 1,\;0<\mu+\nu < 1, \; \kappa = (\mu + \nu)/n}, \\
       B  &=(0,0,0,0,1,1),
 \end{split}
\end{align}
 \item[(iii)] if $\tilde r = \infty$ then the point $\triptrip \in \Sigma_2 \cup C$,
 \begin{align} \label{cond: sigma2}
 \begin{split}
  \Sigma_2 &=\set{\left(\mu,1-\kappa,1,\nu,0,\kappa\right)\;:\;0<\mu, \nu < 1, \; 0 < \mu+\nu < 1, \; \kappa = (\mu + \nu)/n}, \\
  C &=(0,1,1,0,0,0).
  \end{split}
\end{align}
\end{itemize}
\end{defn}

Conditions \eqref{cond: sigma1} and \eqref{cond: sigma2} are sharp which is demonstrated on counterexamples based on Besicovitch sets in Ovcharov \cite{Ovch4}.

\begin{thm} \label{thm: Strich admiss KT}
Let $u$ be the solution to the Cauchy problem for \eqref{eq: KT 1}, \eqref{eq: KT 2}. Then the estimate
\begin{equation} \label{est: Strich admiss KT}
 \normQRP{u}{q}{r}{p} \lesssim \normA{f}{a} + \normQRP{F}{\tilde q'}{\tilde r'}{\tilde p'},
\end{equation}
holds for all $f \in \spaceLa{a}$ and all $F \in L_t^{\tilde q'}L_x^{\tilde r'}L_v^{\tilde p'}$ if and only if $\trip$  and $\tript$ are two KT-admissible exponent triplets and $a=\HM(p,r)=\HM(\tilde p', \tilde r')$, apart from the case in higher dimensions $n>1$ of $\trip$ being an endpoint triplet for which the corresponding estimates  remain unresolved.
\end{thm}

Note that Theorem \ref{thm: Strich admiss KT} allows the second triplet $\tript$ to be endpoint and excludes only the estimates where the first triplet $\trip$ is endpoint. Below we employ the notation $\spaceLqrpf$ for the Lebesgue space $L^q(\R; L^r(\R^n; L^p(V)))$ (or $L^q((0,\infty); L^r(\R^n; L^p(V)))$)  over a finite velocity domain $V \subset \R^n$.

\begin{thm}[Generalized inhomogeneous estimates] \label{thm: glob inhom est KT}
Suppose that $\trip$ and $\tript$ are two jointly KT-acceptable exponent triplets that further satisfy the following conditions
\begin{enumerate}
\renewcommand{\labelenumi}{(\roman{enumi})}
  \item \label{enum: a}  $1 < q,\tilde q < \infty$, $q > \tilde q'$, then the estimate
\begin{equation}\label{eq: Strich inhom KT}
\normQRP{W(t)F}{q}{r}{p} \lesssim \normQRP{F}{\tilde q'}{\tilde r'}{\tilde p'}
\end{equation}
holds for all $F \in \spaceLqrptd$.
  \item $\tilde q = \infty$, $1<q<\infty$, then the estimate
  \begin{align}\label{eq: Strich inhom a}
\normQRP{W(t)F}{q, \infty}{r}{p} &\lesssim \normQRP{F}{1}{\tilde r'}{\tilde p'},\\
\big( \normQRP{W(t)F}{q}{r}{p}  &\lesssim \normQRP{F}{\tilde q'}{\tilde r'}{\tilde p'},\quad q \geq \tilde p'\big) \notag
\end{align}
holds for all $F \in \SpaceLqrp{1}{\tilde r'}{\tilde p'}$.
  \item $q = \infty$, $1<\tilde q<\infty$, then the estimate
  \begin{align}\label{eq: Strich inhom b}
  \normQRP{W(t)F}{\infty}{r}{p} &\lesssim \normQRP{F}{\tilde q', 1}{\tilde r'}{\tilde p'},\\
\big( \normQRP{W(t)F}{q}{r}{p} &\lesssim \normQRP{F}{\tilde q'}{\tilde r'}{\tilde p'},\quad  \tilde q' \leq p \big) \notag
\end{align}
holds for all $F \in \SpaceLqrp{\tilde q', 1}{\tilde r'}{\tilde p'}$ $(F \in \SpaceLqrp{\tilde q'}{\tilde r'}{\tilde p'})$.
  \item \label{enum: d} $1< q,\tilde q < \infty$, $q = \tilde q'$.  Under the assumption of a finite velocity space $V \subset \R^n$ we have that the estimate 
\begin{align} \label{eq: Strich inhom endp}
\normQRPf{W(t)F}{q}{r,q}{P} &\lesssim_V \normQRPfd{F}{q}{r', \tilde q}{P}\\
(\normQRPf{W(t)F}{q}{r}{P}  &\lesssim_V \normQRPfd{F}{q}{r}{P}, \quad q\leq r \text{ and } \tilde q\leq \tilde r \notag)
\end{align}
holds for all $F \in \spaceLqrpftd$, whenever $P, \tilde P$ are such that $1\leq P < p$ and $1 \leq \tilde P < \tilde p$.
\end{enumerate}
Conversely, if estimate (\ref{eq: Strich inhom KT}) holds for all $F \in \spaceLqrpd$, then $\trip$ and $\tript$ must be two jointly KT-acceptable exponent triplets, apart from condition \eqref{cond: glob inhom KT} whose necessity is not fully verified.
\end{thm}

\begin{rem}
As indicated, in the range $q \geq \tilde p'$ estimate \eqref{eq: Strich inhom a} can be strengthen by replacing the Lorentz norm $L^{q, \infty}$ by the Lebesgue norm $L^q$. Analogously, the Lorentz norm $L^{\tilde q', 1}$ in estimate \eqref{eq: Strich inhom b} can be replaced by the Lebesgue norm $L^{\tilde q'}$ in the range $\tilde q' \leq p$. This is proved in Lemma \ref{lem: inhom restr infty}. By the same token, in the range $q\leq r$ and $\tilde q\leq \tilde r$, estimate \eqref{eq: Strich inhom endp} implies its analogue in Lebesgue norms.
\end{rem}

\begin{rem}
If we restrict ourselves to finite time intervals $[0, T]$, we have the continuous embeddings
\begin{align*}
    L^{q,r}([0, T]) \hookrightarrow L^{p}([0, T]),   \quad q > p, \; 1\leq q,p,r\leq \infty,\\
    L^{p}([0, T])   \hookrightarrow L^{q,r}([0, T]), \quad p > q, \; 1\leq q,p,r\leq \infty,
\end{align*}
see \cite[p. 217]{BS}.
For example, let $(\infty, r,p)$ and $\tript$ be such that estimate \eqref{eq: Strich inhom b} holds and let $1\leq \widetilde Q < \tilde q$. Then we have the local inhomogeneous estimate
 \begin{align*}
      \normQRPloc{W(t)F}{\infty}{q}{r}{[0,T]} \lesssim_T \normQRPloc{F}{\widetilde Q'}{\tilde r'}{\tilde p'}{[0,T]}
\end{align*}
for any $0<T<\infty$ and any $F \in \spaceQRPlocd{Q}{r}{p}$.
\end{rem}

\begin{thm}[The Equivalence Theorem] \label{thm: equiv KT}
\mbox{}
\begin{enumerate}
\renewcommand{\labelenumi}{\textbf{\Alph{enumi}.}}
 \item The following three estimates
\begin{align}
  \normQRP{U(t)f}{q}{r}{p}    &\lesssim \normBC{f}{b}{c},   \qquad\qquad \forall f \in \spaceLbc,          \label{est: equiv hom}\\
  \normQRP{W(t)F}{q}{r}{p}    &\lesssim \normQRP{F}{1}{b}{c}, \quad \;\;\forall F \in \SpaceLqrp{1}{b}{c},        \label{est: equiv inhom 1}\\
  \normQRP{W(t)F}{\infty}{b'}{c'} &\lesssim \normQRP{F}{q'}{r'}{p'},   \quad \forall F \in \SpaceLqrp{q'}{r'}{p'}.      \label{est: equiv inhom 2}
\end{align}
are equivalent whenever $1\leq q,r,p, b,c \leq \infty$.
 \item Whenever $b=c=2$  estimate \eqref{est: equiv hom} is equivalent to
\begin{align}
  \normQRP{W(t)F}{q}{r}{r'}       &\lesssim \normQRP{F}{q'}{r'}{r}, \quad  \forall F \in \SpaceLqrp{q'}{r'}{r}.         \label{est: equiv inhom qq'}
\end{align}
\end{enumerate}
\end{thm}
As a direct consequence of Theorem \ref{thm: glob inhom est KT} and the Equivalence Theorem we obtain

\begin{cor}
We have the estimate
\begin{equation}\label{eq: Strich hom mix KT infty}
\normQRP{U(t)f}{q,\infty}{r}{p} \lesssim \normBC{f}{b}{c}
\end{equation}
for all $f \in L^b_xL^c_v$ whenever the exponent 5-vector $\qtet$ satisfies the following conditions
\begin{align}
 \rp q + \frac n r = \frac n b,  &\quad \HM(p,r)=\HM(b,c) \stackrel{def}{=} a, \label{est: gen hom KT 1}\\
  p < &b \leq a \leq c < r, \label{est: gen hom KT 2}\\
  a& \leq r < r^*(c), \label{est: gen hom KT 3}
\end{align}
in the range $1 < q < \infty$, $1 \leq p, \tilde p, r, \tilde r < \infty$. Estimate \eqref{eq: Strich hom mix KT infty} also holds whenever
\begin{equation} \label{est: gen hom KT 4}
 b=c=p=r, \quad q=\infty, \quad \text{(the transport estimate),}
\end{equation}
and whenever $b=p$, $c=r$, and $1/q=n/p-n/r$ (an immediate consequence of the decay estimate \eqref{DispEst}).
\end{cor}

\begin{thm}[Generalized homogeneous estimates] \label{thm: Strich admiss mix KT}
We have the estimate
\begin{equation}\label{eq: Strich hom mix KT}
\normQRP{U(t)f}{q}{r}{p} \lesssim \normBC{f}{b}{c}
\end{equation}
for all $f \in L^b_xL^c_v$ whenever $q \geq c$ and the exponent 5-vector $\qtet$ satisfies \eqref{est: gen hom KT 1} -- \eqref{est: gen hom KT 3} in the range $0 < q,r,p,b,c < \infty$ or condition \eqref{est: gen hom KT 4} in the range $0 < b \leq \infty$. Conversely, if estimate (\ref{eq: Strich hom mix KT}) holds for all $f \in L^b_xL^c_v$ then $\qtet$ must satisfy conditions \eqref{est: gen hom KT 1} and \eqref{est: gen hom KT 2} in the range $1 \leq q < \infty$, $1 \leq p, \tilde p, r, \tilde r < \infty$ or condition \eqref{est: gen hom KT 4} in the range $0 < b \leq \infty$.
\end{thm}

The necessity of condition \eqref{est: gen hom KT 3} remains open in parallel to that of condition \eqref{cond: glob inhom KT} in the setting of the inhomogeneous estimates. Therefore, we cannot exclude the possibility of the existence of more estimates of the form \eqref{eq: Strich hom mix KT} in dimensions $n>1$ in the case of $b \neq c$.

The Equivalence Theorem, part B, together with Theorem \ref{thm: glob inhom est KT}, part (iv), imply the following weaker substitute for the endpoint homogeneous \Str estimate over finite velocity spaces.
\begin{cor} \label{cor: Strich hom endp}
 Let $1\leq P < p^*(a)$,  $(n+1)/n \leq a < \infty$, and let $V \subset \R^n$ be bounded. Then, the following estimate
\begin{align} \label{eq: Strich hom endp}
\normQRPf{U(t)f}{a}{r^*(a)}{P} \lesssim_V \normLa{f}{a},
\end{align}
holds for all $f \in \spaceLa{a}$.
\end{cor}

\section{General introductory remarks}

We owe the reader an explanation why our results do not follow from earlier works on \Str estimates. As it is well-known these estimates follow from two main ingredients, the decay and the energy estimates. Besides these, in the context of the KT equation, it is also necessary to assume a further structure condition that greatly increases the range of estimates we may prove. Consider the decay estimate
\begin{equation} \label{eq: disp est 1}
  \norm{U(t)f}_{L^{\infty}_xL^{1}_v} \lesssim \frac 1 {\abs{t}^{n}}\norm{f}_{L^{1}_xL^{\infty}_v}, \quad \forall t \in \R,
\end{equation}
for the KT equation. Note that the mixed Lebesgue norm in \eqref{eq: disp est 1} creates difficulties in interpolation if one uses the real method. Moreover, the general results of Keel and Tao \cite{KT} and Taggart \cite{Tg} are based on the real method and if applied to the present context produce \Str estimates in non-Lebesgue norms. Additional complication arises from the fact that the KT propagator $U(t)$ preserves a whole family of Lebesgue norms
\begin{equation} \label{eq: tran est 1}
  \normLa{U(t)f}{a} = \normLa{f}{a}, \quad \forall t \in \R, \; 0 < a \leq\infty,
\end{equation}
and not just the $L^2$-norm (corresponding to the energy estimate in other contexts). To mark the different nature of estimate \eqref{eq: tran est 1} we shall call it the transport estimate and any class $\spaceLa{a}$ for $1\leq a \leq \infty$ we shall call a transport class. The transport estimate is a consequence of the special case $a=2$ in \eqref{eq: tran est 1} and the following invariance of the homogeneous KT equation
\begin{equation}\label{power transform}
f \rightarrow f^\alpha, \quad U(t)f \rightarrow (U(t)f)^\alpha, \quad 0<\alpha<\infty.
\end{equation}
Furthermore, this invariance allows us to prove new homogeneous \Str estimates from already proven ones. In fact, the exponents in
\begin{align*}
  \normQRP{U(t)f}{q}{r}{p}  \lesssim \normBC{f}{b}{c},  \quad \forall f \in \spaceLbc,
\end{align*}
 transform according to the rule
\begin{equation} \label{correspondence alpha}
(q, r, p, b, c) \rightarrow (\alpha q, \alpha r, \alpha p, \alpha b, \alpha c)
\end{equation}
and any two estimates whose exponents are related in such a way are equivalent. Note that there is no such convenient tool in the inhomogeneous setting.

To summarize, the \Str estimates that we shall prove in the sequel are consequences of the decay estimate \eqref{eq: disp est 1}, the transport estimate \eqref{eq: tran est 1}, and the structural assumption \eqref{power transform}.

We would like next to highlight some special advances that we make in the present work. Most of all, we study the equivalence between different types of \Str estimates. One such result is the fact that in the context of the KT equation the \Str estimates for the operator $W(t)$ and that of $TT^*$,
\[
TT^*F = \Int U(t-s) F(s) ds,
\]
are equivalent. This greatly simplifies the use of duality arguments and we do not any longer need the Christ-Kiselev lemma in order to deduce the inhomogeneous \Str estimates via the corresponding estimates for $TT^*$.

We also show that any homogeneous estimate has corresponding inhomogeneous estimates to which it is equivalent e.g.
\begin{align}
  \normQRP{U(t)f}{q}{r}{p}    &\lesssim \normBC{f}{b}{c},   \qquad\qquad \forall f \in \spaceLbc,        \label{est: equiv hom intro}\\
  \normQRP{W(t)F}{q}{r}{p}    &\lesssim \normQRP{F}{1}{b}{c} \quad \;\;\forall F \in \SpaceLqrp{1}{b}{c}.        \label{est: equiv inhom 1 intro}
\end{align}
are equivalent, and more generally see Theorem \ref{thm: equiv KT}.

One possible application of this equivalence is in the study of the range of estimates \eqref{est: equiv hom intro}. Since the proof of these estimates in the present context is an entirely new result, we shall give an example from the context of the \Sch equation. The estimate
\begin{align} \label{eq: Sch}
  \normQR{U_s(t)f}{q}{r}    \lesssim \normLx{f}{p},   \qquad\qquad \forall f \in \spaceLx{p},
\end{align}
where by $U_s(t)$  we have denoted the \Sch propagator, was investigated by T. Kato \cite{Ka} (1994) for $1<p \leq 2$. We obtain a larger range of such estimates in higher dimensions $n>2$, see our PhD Thesis \cite{Ovch3}. This improvement is essentially due to the fact that our approach benefits from the more recent advances introduced by Keel and Tao \cite{KT} and Foschi \cite{F} in the inhomogeneous setting and, of course, the implication of \eqref{est: equiv hom intro} by \eqref{est: equiv inhom 1 intro}.

The last special result to be considered here is concerned with the endpoint \Str estimates for the KT equation in higher dimensions as in Theorem \ref{thm: glob inhom est KT}, part (iv). We prove there a class of estimates with a loss of integrability that can be made arbitrary small compared to the original endpoint estimates. However, our estimates are given entirely in terms of Lebesgue norms which is useful in applications. Furthermore, we give a counterexample showing that there does not exist a family of perturbed local estimates in a ``full neighborhood'' around any given endpoint estimate. The existence of the latter is required by the methods of Keel and Tao \cite{KT} and Foschi \cite{F}, and hence why it has not been yet possible to resolve in the positive the endpoint estimates of the considered type.

The different cases in  Theorem \ref{thm: glob inhom est KT} can be visualized quite easily. Let us first remember that the Lebesgue space $L^p$ is best seen as a ``function" of $1/p$ rather than $p$ in the context of interpolation. Therefore, the range of validity of the estimate \eqref{eq: Strich inhom KT} corresponds to a region in $\R^6$ of points with coordinates $(1/q, 1/r, 1/p, 1/{\tilde q}, 1/{\tilde r}, 1/{\tilde p})$. The projection of that region over the $1/q$-$1/{\tilde q}$-plane is visualized in Figure \ref{fig: 1}.

\begin{figure}[htp]
\centering
 \resizebox{4.5cm}{!}{\includegraphics[viewport=7cm 19.5cm 13.5cm 25.5cm, clip]{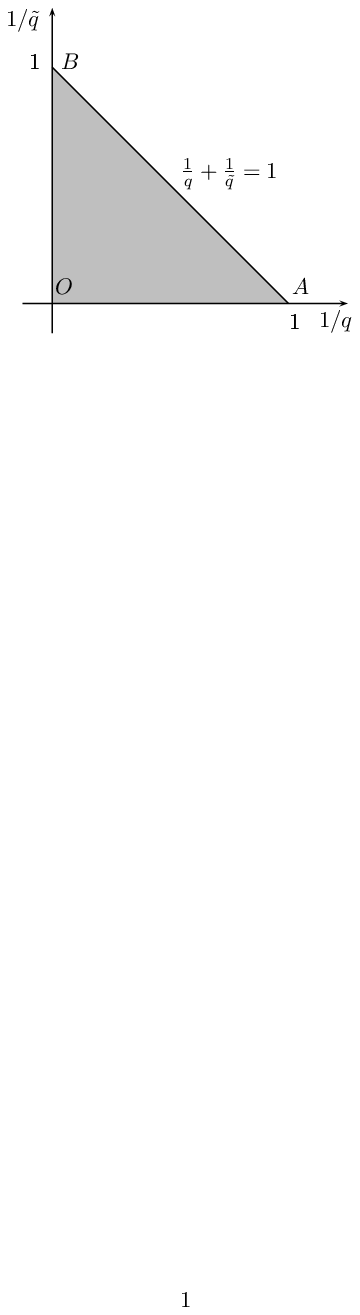}} 
\caption{Acceptable range of $(1/q, 1/{\tilde q)}$.}\label{fig: 1}
\end{figure}

The inner part of $\Delta OAB$ corresponds to the non-endpoint inhomogeneous estimates, while its three sides correspond to the endpoint inhomogeneous estimates. In the context of Theorem \ref{thm: glob inhom est KT}, the inner part of $\Delta OAB$ corresponds to part (i), the cathetus $OA$ - to part (ii), the cathetus $OB$ - to part (iii), and the hypotenuse $AB$ - to part (iv). The inhomogeneous estimates can be put into three groups in rising order of difficulty: the inner part of $\Delta OAB$, the two catheti $OA$ and $OB$, and the hypotenuse $AB$.

Note that in one spatial dimension condition \eqref{cond: glob inhom KT} is void  and thus there the complete range of validity of the \Str estimates for the KT equation is now known. In higher dimensions, however, the necessity (sharpness) of this condition is open. A similar question one encounters in other contexts, see e.g. Foschi \cite{F} in the context of the \Sch equation.

Before we end this section we remark that the estimates that we prove remain valid for more general domains than those considered in the definition of equation \eqref{eq: KT 1}. For example, the domain of $t$ may be any interval $I \subseteq \R$, and the domain of $v$ may be any measurable set $V \subseteq \R^n$. The claims follow from the fact that the transport and the dispersive estimate for the KT equation remain valid for these domains, as can one easily see by a simple modification of the proofs of Lemmas \ref{lem: decay est} and \ref{lem: tran est}.

The remainder of the paper is organized as follows. In the next section we give some auxiliary facts about the KT equation and in Section \ref{sect: duality} we present the $TT^*$ method and some other duality arguments including the proof of the Equivalence Theorem \ref{thm: equiv KT}. The proof of the \Str estimates for the Cauchy problem (Theorem \ref{thm: Strich admiss KT}) is given in Section \ref{sect: proof 1}. The local inhomogeneous \Str estimates are derived in Section \ref{sect: loc inhom est KT}. The generalized \Str estimates are proved in Section \ref{sect: proof glob inhom}. In Section \ref{sect: counterex} we  show sharpness of the estimates that we prove by means of counterexamples. We finish the paper with Section \ref{sect: que} where we list some still unanswered questions regarding the \Str estimates for the KT equation.

\section{Some properties of the kinetic transport equation} \label{sect: properties}
\begin{lem}[The dispersive estimate \cite{P}] \label{lem: decay est}
The kinetic transport evolution group $U(t)$ obeys the estimate
\begin{equation} \label{eq: disp est 2}
  \norm{U(t)f}_{L^{\infty}_xL^{1}_v} \leq \frac 1 {\abs{t}^{n}}\norm{f}_{L^{1}_xL^{\infty}_v},
 \end{equation}
for all $f \in {L^{1}_xL^{\infty}_v}$.
\end{lem}
\begin{proof}
\begin{equation*}
 \begin{split}
  \int_{\R^n} \abs{U(t)f} dv = \int_{\R^n} \abs{f(x-tv,v)}dv \leq \int_{\R^n} \sup_{y \in \R^n} \abs{f(x-tv,y)} dv \\
  \leq \rp {\abs{t}^n} \int_{\R^n} \sup_{y \in \R^n} \abs{f(z,y)} dz = \rp {\abs{t}^n} \normBC{f}{1}{\infty}.
 \end{split}
\end{equation*}
\end{proof}

\begin{lem}[The transport estimate] \label{lem: tran est}
The kinetic transport evolution group $U(t)$ obeys the estimate
\begin{equation} \label{eq: tran est 2}
  \normQRP{U(t)f}{\infty}{a}{a} \leq \normA{f}{a}, \quad 0<a\leq \infty,
 \end{equation}
for all $f \in L^{a}_{x,v}$.
\end{lem}
\begin{proof}
 Trivial.
\end{proof}

\begin{cor}[The decay estimate] \label{cor: dec est}
The kinetic transport evolution group $U(t)$ obeys the estimate
 \begin{equation}\label{DispEst}
   \norm{U(t)f}_{L^r_xL^p_v} \leq \frac 1 {\abs{t}^{n\left(\frac 1 p - \frac 1 r\right)}}\norm{f}_{L^p_xL^r_v}, \quad 1 \leq p \leq r \leq \infty,
 \end{equation}
for all $f \in {L^{p}_xL^{r}_v}$.
\end{cor}
\begin{proof}
Complex interpolation between the decay estimate (\ref{eq: disp est 2}) and the two transport estimates (\ref{eq: tran est 2}) with $a=1$ and $a=\infty$.
\end{proof}

\begin{lem} \label{lem: adjoint}
The formal adjoint to $U(t)$ is the operator $U^*(t)=U(-t)$.
\end{lem}
\begin{proof}
We denote by $\langle \cdot, \cdot \rangle$ the scalar product on $L^2(\R^{2n})$. Thus,
\begin{align*}
\langle U(t)f,g\rangle &= \Int f(x-tv,v) \overline{g(x,v)}dxdv\\
          &=\Int f(y,v) \overline{g(y+tv,v)}dydv=\langle f,U(-t)g\rangle,
\end{align*}
where we have made the substitution $y=x-tv$.
\end{proof}

\begin{lem}[Scaling properties of $U(t)$ and $W(t)$] \label{lem: scaling}
The evolution operators $U(t)$ and $W(t)$ enjoy the following scaling properties
\begin{align*}
U(t)f_\lambda = f\left({x}/{\lambda} - {t v}/{\lambda} , v\right)= \{U(\cdot)f\}(t/\lambda, x/\lambda, v),\\
\text{where  } f_\lambda(x,v) = f\left(x/{\lambda}, v\right),\\
U(t)f_\lambda = f\left(x / {\lambda} - {t v}/{\lambda} , {v}/{\lambda}\right)= \{U(\cdot)f\}(t, x/\lambda, v/\lambda),\\
\text{where  } f_\lambda(x,v) = f\left(x/{\lambda}, v/\lambda \right),\\
W(t)F_\lambda = \lambda \int_{-\infty}^{t / \lambda}F\left(s, x / {\lambda} - \left(t / \lambda - s \right) v, v\right)ds =  \lambda \{W(\cdot)F\}\left(t/ \lambda,x/ \lambda,v\right),\\
\text{where  } F_\lambda(t,x,v) = F(t/\lambda, x/\lambda, v),\\
W(t)F_\lambda = \intI F\left(s, x / {\lambda} - (t-s) v / {\lambda} , v / \lambda \right)ds= \{W(\cdot)F\}(t,x/\lambda,v/\lambda),\\
\text{where  } F_\lambda(t,x,v) = F\left(t, x/ \lambda, v/ \lambda\right).
\end{align*}
\end{lem}
\begin{proof}
 Direct inspection.
\end{proof}

\section{Duality and the $TT^*$-principle} \label{sect: duality}

At the heart of the proof of \Str estimates lie duality arguments. In this section we introduce the main elements of all duality constructions in later proofs.
\subsection{Basics.}
Let us consider the operator
\begin{align*}
    T: \spaceLtwo \rightarrow \SpaceLqrp{q}{r}{r'}, \quad \{Tf\}(t,x,v)=f(x-tv,v)=U(t)f,
\end{align*}
for some Lebesgue exponents $2\leq q,r \leq \infty$.
Its formal adjoint is the $L^2$-valued integral
\begin{align*}
    T^*:\SpaceLqrp{q'}{r'}{r} \rightarrow \spaceLtwo, \quad \{T^*F\}(x,v)=\Int F(s, x + sv,v)ds.
\end{align*}
The composition  of the two has the form
\begin{gather*}
    TT^*: \SpaceLqrp{q'}{r'}{r} \rightarrow \SpaceLqrp{q}{r}{r'}, \\
                          \{TT^*F\}(t,x,v)=\Int F(s, x - (t-s)v,v)ds = \Int U(t-s)F(s)ds.
\end{gather*}
In view of the $TT^*$-principle, see e.g. \cite[p. 113]{S}, $T$ and $TT^*$ are equally bounded with $\norm{T}^2=\norm{TT^*}$. Thus, the following two estimates are equivalent
\begin{align}
 \normQRP{Tf}{q}{r}{r'}    &\leq C \normA{f}{2}, \quad\qquad \forall f \in \spaceLa{2}, \label{est: T}\\
 \normQRP{TT^*F}{q}{r}{r'} &\leq C^2 \normQRP{F}{q'}{r'}{r}, \quad \forall F \in \spaceLqrpd, \label{est: TT*}
\end{align}
where $C=\norm{T}$. We shall call \eqref{est: TT*} the symmetric $TT^*$-estimate. By $\langle \cdot, \cdot \rangle$ we denote the duality pairing on $\R^{2n}$
\[
\langle f, g \rangle = \int_{R^{2n}}f(x,v)g(x,v)dxdv
\]
for the (mixed) Lebesgue spaces. Thus, in bilinear formulation \eqref{est: TT*} reads
\begin{align*}
    \abs{\Int \langle \{TT^*F\}(t), G(t) \rangle dt} \leq C^2 \normQRP{F}{q'}{r'}{r}\normQRP{G}{q'}{r'}{r}.
\end{align*}
In view of Lemma \ref{lem: adjoint}, this is equivalent to
\begin{align*}
    \abs{\Int\Int \langle U(s)^*F, U(t)^*G \rangle dsdt} \leq C^2 \normQRP{F}{q'}{r'}{r}\normQRP{G}{q'}{r'}{r},
\end{align*}
$\forall F, \; \forall G \in \SpaceLqrp{q'}{r'}{r}$.
In \cite{KT} Keel and Tao noted that by symmetry, i.e. by changing the roles of $F$ and $G$, the latter estimate is always implied by the estimate
\begin{align} \label{eq: B 1}
   \abs{B(F,G)} \leq C^2 \normQRP{F}{q'}{r'}{r}\normQRP{G}{q'}{r'}{r}, \quad \forall F, \; \forall G \in \SpaceLqrp{q'}{r'}{r},
\end{align}
for the bilinear form
\[
  B(F,G) = \iint_{s<t}  \langle U(s)^*F, U(t)^*G \rangle dsdt.
\]
Furthermore, we shall prove in the special context of the KT equation that these two estimates are in fact equivalent, see Lemma \ref{lem: TT* KT}.

We now consider the inhomogeneous estimates. Suppose that $\trip$ and $\tript$ are two exponent triplets such that
\begin{align*}
\normQRP{Tf}{q}{r}{p}    &\leq C \normA{f}{a}, \quad\qquad \forall f \in \spaceLa{a},\\
\normQRP{Tf}{\tilde q}{\tilde r}{\tilde p}    &\leq C \normA{f}{a'}, \quad\qquad \forall f \in \spaceLa{a'},
\end{align*}
for some $1\leq a \leq \infty$. The composition
\[
  \spaceLqrptd \stackrel{T^*}{\To} \spaceLa{a} \stackrel{T}{\To} \spaceLqrp
\]
is bounded and thus
\begin{align} \label{est: TT* 2}
 \normQRP{TT^*F}{q}{r}{p} \leq C^2 \normQRP{F}{\tilde q'}{\tilde r'}{ \tilde p'}, \quad \forall F \in \spaceLqrptd.
\end{align}
This is the general non-symmetric $TT^*$-estimate for the KT equation. It does not any longer imply boundedness for the operator $T$, but as we shall see next, it implies boundedness for the operator $W(t)F$.

\begin{lem} \label{lem: TT* KT}
In the context of the KT equation the non-symmetric $TT^*$-estimate and the inhomogeneous \Str estimates are equivalent, i.e
\eqref{est: TT* 2} is equivalent to
\begin{align}
  \normQRP{W(t)F}{q}{r}{p} \leq C^2 \normQRP{F}{\tilde q'}{\tilde r'}{ \tilde p'}, \quad \forall F \in \spaceLqrptd.\label{eq: TT W}
\end{align}
\end{lem}
\begin{proof}
(i) In one direction the claim follows immediately from
\[
 \abs{W(t)F} \leq TT^*\abs{F}.
\]

(ii) In the other direction we have the following argument. It is easy to see that by duality \eqref{eq: TT W} is equivalent to
\begin{align} \label{eq: B 2}
   \abs{B(F,G)} \lesssim \normQRP{F}{\tilde q'}{\tilde r'}{\tilde p'} & \normQRP{G}{q'}{r'}{p'}, \\
   &\forall F \in \spaceLqrptd ,\; \forall G \in \spaceLqrpd. \notag
\end{align}
By making the substitution $\sigma=-s,\,\tau=-t$ in the definition of $B(F,G)$ we get
\begin{equation*}
 \begin{split}
   B(F,G) = \iint_{\tau < \sigma}  \langle U(-\sigma)^*F, U(-\tau)^*G \rangle d\tau d\sigma.
 \end{split}
\end{equation*}
 The integral in the line above can be written as
\[
  (-1)^n  \iint_{\tau < \sigma}  \langle U(\sigma)^*F', U(\tau)^*G' \rangle d\tau d \sigma \stackrel{def}{=} (-1)^n B'(F',G'),
\]
by making the substitution $x \To -x$ and setting $F'(t,x,v)=F(-t,-x,v)$, $G'(t,x,v)$ = $G(-t,-x,v)$. Thus the boundedness of the bilinear form $B(\cdot, \cdot)$ implies the boundedness of the bilinear form $B'(\cdot, \cdot)$ on the same spaces. The claim follows from the fact that the boundedness of $TT^*$ is equivalent to that of the bilinear form $B+B'$.
\end{proof}

For convenience we summarize the bilinear formulation of the duality arguments of this paragraph in
\begin{lem}\label{lem: bilinear}
\mbox{}
\begin{enumerate}
\renewcommand{\labelenumi}{(\roman{enumi})}
\item The boundedness of the operator $T : \spaceLa{2} \rightarrow  \SpaceLqrp{q}{r}{r'}$ of the form $Tf=U(t)f$ is equivalent to the boundedness of the bilinear mapping $B: \SpaceLqrp{q'}{r'}{r}$ $\times$ $\SpaceLqrp{q'}{r'}{r}  \rightarrow \C$.

\item The boundedness of the operators $W(t): \SpaceLqrp{\tilde q'}{\tilde r'}{\tilde p'} \rightarrow \SpaceLqrp{q}{r}{p}$ and $TT^*: \SpaceLqrp{\tilde q'}{\tilde r'}{\tilde p'} \rightarrow \SpaceLqrp{q}{r}{p}$ is equivalent to that of the bilinear mapping $B: \SpaceLqrp{\tilde q'}{\tilde r'}{\tilde p'} \times \SpaceLqrp{q'}{r'}{p'} \rightarrow \C$.
\end{enumerate}
\end{lem}
\subsection{Equivalence of \Str estimates.}
\begin{lem}[The Duality lemma] \label{lem: dual}
The following two estimates for $W(t)$ are equivalent
\begin{align*}
  \normQRP{W(t)F}{q}{r}{p}   &\lesssim \normQRP{F}{\tilde q'}{\tilde r'}{\tilde p'}, \quad \forall F \in \spaceLqrptd,\\
  \normQRP{W(t)F}{\tilde q}{\tilde r}{\tilde p} &\lesssim \normQRP{F}{q'}{r'}{p'}, \quad \forall F \in \spaceLqrpd,
\end{align*}
for $1\leq p,q\leq \infty$.
\end{lem}
\begin{proof}
It follows immediately from the fact the boundedness of $W(t)$ is equivalent to that of the symmetric operator $TT^*$ on the considered spaces.
\end{proof}

\begin{thm}[The Equivalence Theorem] \label{thm: equiv KT 2}
\mbox{}
\begin{enumerate}
\renewcommand{\labelenumi}{({\roman{enumi})}}
 \item The following three estimates
\begin{align}
  \normQRP{U(t)f}{q}{r}{p}    &\lesssim \normBC{f}{b}{c},   \qquad\qquad \forall f \in \spaceLbc,          \label{est: equiv hom v2}\\
  \normQRP{W(t)F}{q}{r}{p}    &\lesssim \normQRP{F}{1}{b}{c}, \quad \;\;\forall F \in \SpaceLqrp{1}{b}{c},        \label{est: equiv inhom 1 v2}\\
  \normQRP{W(t)F}{\infty}{b'}{c'} &\lesssim \normQRP{F}{q'}{r'}{p'},   \quad \forall F \in \SpaceLqrp{q'}{r'}{p'}.      \label{est: equiv inhom 2 v2}
\end{align}
are equivalent whenever  $1\leq q,r,p, b,c \leq \infty$.
 \item Whenever $b=c=2$  estimate \eqref{est: equiv hom v2} is equivalent to
\begin{align}
  \normQRP{W(t)F}{q}{r}{r'}       &\lesssim \normQRP{F}{q'}{r'}{r}, \quad  \forall F \in \SpaceLqrp{q'}{r'}{r}.         \label{est: equiv inhom qq' v2}
\end{align}
\end{enumerate}
\end{thm}
\begin{proof}
(i) The homogeneous estimate \eqref{est: equiv hom v2} trivially implies the first inhomogeneous estimate \eqref{est: equiv inhom 1 v2}. In view of the Duality lemma \ref{lem: dual}, the two inhomogeneous estimates \eqref{est: equiv inhom 1 v2} and \eqref{est: equiv inhom 2 v2} are equivalent. All it remains to show is that \eqref{est: equiv inhom 1 v2} implies \eqref{est: equiv hom v2}.

Let us first give a short formal proof. We choose $F(t) = \delta(t) f$ where $\delta(t)$ is the delta function on $\R$ and $f \in \spaceLbc$. Consequently, $ W(t)[\delta(\cdot)f]=U(t)f$ and thus
\begin{align*}
 \normQRP{U(t)f}{q}{r}{p} \lesssim \normQRP{\delta(t)f}{1}{b}{c}= \normLbc{f}{b}{c}.
\end{align*}

To make that rigorous instead of $\delta(t)$ we consider a smooth approximation of the identity $\delta_{\epsilon}(t)$, for $\epsilon >0$. So we are given the estimate
\begin{align*}
 \abs{B(F,G)} \lesssim \normQRP{F}{1}{b}{c}\normQRP{G}{q'}{r'}{p'}, \quad  \forall F \in \SpaceLqrp{1}{b}{c}, \; \forall G \in \SpaceLqrp{q'}{r'}{r}.
\end{align*}
It would be enough to prove that
\[
 B(\delta_{\epsilon}*f, G) \To \Int \langle U(t)f, G \rangle dt
\]
since we have that $\normQRP{\delta_{\epsilon}*f}{1}{b}{c}$ $=$ $\normBC{f}{b}{c}$, for any $\epsilon > 0$. To prove the limit it would be enough to consider only nonnegative functions $f$ and smooth nonnegative functions $G$ of compact support in $t \geq 0$. For $t>0$ we have
\begin{align*}
 \begin{split}
 B(\delta_{\epsilon}f, G) &= \int \intI \langle \delta_{\epsilon}(s)*f, U(s-t)G \rangle dsdt \\
    &= \Int \Int \delta_{\epsilon}(t-s) \langle f, U(-s)G \rangle dsdt \To \int \langle U(t)f, G \rangle dt.
 \end{split}
\end{align*}
The last limit is justified by the fact that the function $h(s)$ $=$ $\langle f, U(-s)G \rangle $ is continuous and thus $\delta_{\epsilon}*h(t)$ $\To$ $h(t)$ as $\epsilon$ $\To$ $0$. Then, in view of Fatou's lemma,
\begin{align*}
 \begin{split}
  \int \langle U(t)f, G \rangle dt &\leq \liminf_{\epsilon \rightarrow 0} B(\delta_{\epsilon}f, G) \\
    &\lesssim \normLbc{f}{b}{c}\normQRP{G}{q'}{r'}{p'}, \quad \forall f \in \spaceLbc.
 \end{split}
\end{align*}

(ii) This follows directly from Lemma \ref{lem: bilinear}.
\end{proof}

\begin{rem}\label{rem: Equiv thm KT}
We shall need a slightly more general form of the the Equivalence Theorem in the sequel, where the temporal norm is the Lorentz $L^{q, s}$-norm in time. In fact, we shall be only interested in the case when $s=q$ and thus $L^{q, q}$ is equivalent to $L^{q}$, and in the cases when $s=1$ or $s=\infty$. The proof is almost identical and will be omitted.
\end{rem}

\subsection{Local in time decompositions and scaling}

We shall further refine our main tool which is Lemma \ref{lem: bilinear} by introducing a temporal localization of the bilinear form $B$. More precisely, $B$ shall be decomposed into a sum of scaling invariant dyadic pieces induced by a Whitney's dyadic decomposition applied on the domain $\Omega = \{(t,s)|s<t\}$ used to define $B$.

\begin{defn} We call any positive integer that is a power of two a dyadic number. Furthermore, we call a square $Q$ in $\R^2$ dyadic if its side length is a dyadic number and the coordinates of its vertices are integer multiples of dyadic numbers.
\end{defn}

We apply  Whitney's dyadic decomposition on $\Omega$ and obtain the family $\mathcal O$ of essentially disjoint dyadic squares $Q$ (by that we mean that overlapping on the sides is still possible) such that the distance between any square $Q \in \mathcal{O}$ and the boundary of $\Omega$ ($\{(t,s) | t=s \}$) is approximately proportional to the diameter of $Q$. This is immediately obvious in Figure \ref{fig: Whitney}.
\begin{figure}[htp]
\centering
\resizebox{5cm}{!}{\includegraphics[viewport=0 0 285pt 285pt, clip]{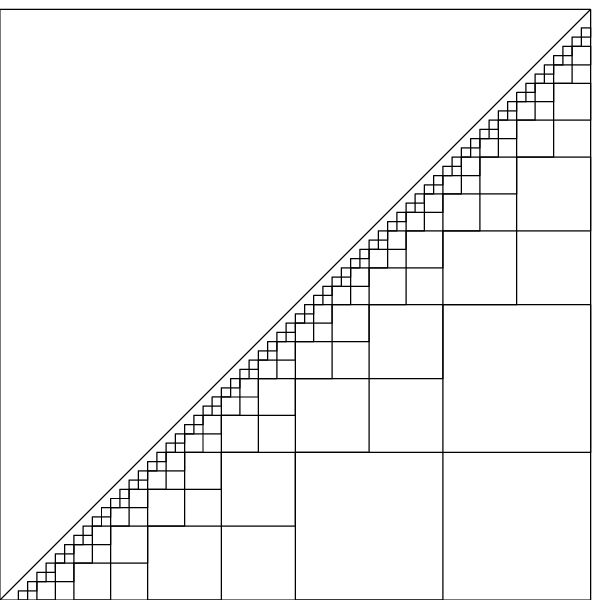}}
\caption{Whitney's decomposition for the region $s < t$}\label{fig: Whitney}
\end{figure}
By $\Ol$ we denote the collection of all squares in $\mathcal{O}$ whose side length is $\lambda$. Thus we obtain the representation
\[
B(F, G) = \sum_\lambda \sum_{Q \in \Ol} B_Q(F, G), \quad  \text{where } \Omega = \bigcup_\lambda \bigcup_{Q \in \Ol},
\]
and
\begin{equation}\label{BQ}
B_Q(F, G) = \iint_{Q} \langle U^*(s)F(s), U^*(t)G(t) \rangle ds dt.
\end{equation}
Furthermore, whenever $Q = J \times I$ and $Q \in \Ol$ we have
\begin{equation}\label{lambdaIJ}
\lambda = \abs I = \abs J \sim \mathrm{dist}(\Omega, \partial\Omega)\sim \mathrm{dist}(I, J).
\end{equation}
The localized bilinear operator $B_Q$ scales in the following way
\begin{equation}\label{BilStrichQ}
\abs{B_{Q}(F, G)} \lesssim \lambda^{\beta (q,r,{\tilde{q}},{\tilde{r}})}  \normQRPDloc{F}{q}{r}{p}{J} \normQRPloc{G}{q'}{r'}{p'}{I},
\end{equation}
where
\begin{align} \label{eq: beta KT}
 \beta(q,r,\tilde q, \tilde r) = \bbeta.
\end{align}
The range of the Lebesgue exponents $\triptrip$ for which we can prove that \eqref{BilStrichQ} holds for every $Q \in \Ol$ will be presented in Lemma \ref{lem: local inhom estim KT}. Property \eqref{BilStrichQ} implies
\begin{lem}\label{cor: 1/q+}
If $\rp q + \rp {\tilde{q}} \leq 1$, then                         
\begin{equation}\label{cor1Est}
\sum_{Q \in \Ol}\abs{B_{Q}(F, G)} \lesssim  \lambda^{\beta(q, r ,{\tilde{q}}, {\tilde{ r}})} \normQRPDloc{F}{q}{r}{p}{\R} \normQRPloc{G}{q'}{r'}{p'}{\R}.
\end{equation}
for every $F \in \SpaceLqrptdloc{q}{r}{p}{\R}$, and every $G \in \SpaceLqrploc{q'}{r'}{p'}{\R}$.
\end{lem}
\begin{proof}
In view of (\ref{BilStrichQ})
\begin{equation*}
 \sum_{Q \in \Ol} \abs{B_Q(F, G)} \lesssim \lambda^{\beta(q, r ,{\tilde{q}}, {\tilde{ r}})} \sum_{Q \in \Ol, \, Q=J\times I} \normQRPDloc{F}{q}{r}{p}{J} \normQRPloc{G}{q'}{r'}{p'}{I}.\qquad
\end{equation*}
The claim now follows immediately from Lemma \ref{lem: 1/p+} below.
\end{proof}

\begin{lem}\label{lem: 1/p+}
Suppose $\rp p + \rp {\tilde{p}} \geq 1$. Then                                    
\[
\sum_{Q \in \Ol, \ Q=J\times I}\norm{f}_{L^{\tilde{p}}(J)}\norm{g}_{L^{p}(I)}\leq
\norm{f}_{L^{\tilde{p}}(\R)}\norm{g}_{L^{p}(\R)}.
\]
\end{lem}
\begin{proof}  The lemma follows directly from the inequality
\[
\sum_{j} \abs{a_j b_j} \leq \left( \sum_j |a_j|^{\tilde p} \right)^{\frac 1 {\tilde p}}
\left( \sum_j |b_j|^{p} \right)^{\frac 1 {p}},
\]
which holds in the range $\rp p + \rp {\tilde{p}} \geq 1$, and the fact that for each dyadic interval $I$ there are at most two dyadic squares in $\Ol$ with side $I$.
\end{proof}

We now introduce the bilinear operator $A: \SpaceLqrp{\tilde q'}{\tilde r'}{\tilde p'} \times \SpaceLqrp{q'}{r'}{p'} \rightarrow l^\infty_s$, (for a definition of $l^\infty_s$ see below), defined by the formula
\begin{align*}
 A(F, G) = \left\{ b_\lambda \right\}_{\lambda \in 2^{\Z}} = \left\{ \sum_{Q \in \Ol} \abs{B_Q(F, G)} \right\}_{\lambda \in 2^{\Z}}.
\end{align*}
For instance, this operator is bounded whenever we have property \eqref{BilStrichQ} (see Lemma \ref{lem: local inhom estim KT}), $q\geq \tilde q'$, and $s= - \beta(q,r,\tilde q, \tilde r)$. Clearly, the boundedness of  $A: \SpaceLqrp{\tilde q'}{\tilde r'}{\tilde p'}$ $\times$ $\SpaceLqrp{q'}{r'}{p'} \rightarrow l^1$ implies the boundedness of $B: \SpaceLqrp{\tilde q'}{\tilde r'}{\tilde p'}$ $\times$ $\SpaceLqrp{q'}{r'}{p'} \rightarrow \C$. Thus, in view of Lemma \ref{lem: bilinear}, the estimate
\[
 \norm{\{ b_\lambda\}}_{l^1} \lesssim \normQRPD{F}{q}{r}{p} \normQRP{G}{q'}{r'}{p'}, \quad \forall F \in \spaceLqrptd,\; \forall G \in \spaceLqrp,
\]
implies the boundedness of $W(t): \spaceLqrptd \rightarrow \spaceLqrp$ . We summarize this fact in
\begin{lem} \label{lem: bil A}
The boundedness of the bilinear operator $A: \SpaceLqrp{\tilde q'}{\tilde r'}{\tilde p'} \times \SpaceLqrp{q'}{r'}{p'} \rightarrow l^1$ implies the inhomogeneous \Str estimate
\begin{align*}
\normQRP{W(t)F}{q}{r}{p} &\lesssim \normQRPD{F}{q}{r}{p}, \quad \forall F \in \spaceLqrptd.
\end{align*}
\end{lem}

In our proofs of the inhomogeneous \Str estimates for the KT equation we shall be making a repeated use of this lemma. We shall also need a number of standard results from the theory of Interpolation Spaces which are given in the remaining part of this paragraph. By $L^p=L^p(X; \mathcal B)$ and $L^{p,q}=L^{p,q}(X; \mathcal B) \label{sym: Lorentz}$  we denote the  Lebesgue space and the Lorentz space respectively of vector-valued functions that map a fixed measure space $(X,d\mu)$ to a fixed Banach space $\mathcal B$.

\begin{lem}[{see \cite[ p. 113]{BL}}] \label{prop: real interp 2}
Suppose that $0 < p_0,\, p_1,\, q_0,\, q_1 \leq \infty$, $0< \theta < 1$, and $p_0 \neq p_1$. Then
\begin{equation*}
  \left( L^{p_0, q_0}, L^{p_1, q_1} \right)_{\theta, q} = L^{p, q}, 
\end{equation*}
where $1/p = (1 - \theta)/p_0 + \theta/p_1$.
\end{lem}
Suppose that $\mathcal B_0$ and $\mathcal B_1$ are two Banach spaces that are compatible for interpolation.

\begin{lem}[{see the Appendix of \cite{CN}}]  \label{prop: real interp 1}
For every $1 \leq p_0, p_1 < \infty$, $0 < \theta < 1$, $1/p = (1 - \theta)/p_0 + \theta/p_1$ and $p \leq q$ we have
\begin{equation*}
 L^p (X; (\mathcal B_0, \mathcal B_1)_{\theta ,q}) \hookrightarrow (L^{p_0}(X;\mathcal B_0), L^{p_1}(X; \mathcal B_1))_{\theta,q}.
\end{equation*}
\end{lem}
Denote by $l^p_s \label{sym: lp}$ the space of number sequences with a norm
\begin{align*}
 \norm{\{a\}_{j \in \Z}}_{l^p_s}      &= \left(2^{js}\abs{a_j}^p \right)^{1/p}, \quad 1\leq p<\infty,\\
 \norm{\{a\}_{j \in \Z}}_{l^\infty_s} & = \sup_{j \in \Z} 2^{js}\abs{a_j},  \qquad p=\infty.
\end{align*}

\begin{lem}[See Theorem 5.6.1 in \cite{BL}] \label{lem: interp l}
We have the identity
\begin{align*}
 \left(l^\infty_{s_0}, l^\infty_{s_1} \right)_{\theta, 1} = l^1_{s},
\end{align*}
where $s_0, s_1 \in \R$, $s_0 \neq s_1$ and $s=(1-\theta)s_0+\theta s_1$.
\end{lem}

\begin{lem}[See p. 76 in \cite{BL}] \label{lem: interp T1}
Suppose that $(\mathcal A_0, \mathcal A_1)$, $(\mathcal B_0, \mathcal B_1)$, and $(\mathcal C_0, \mathcal C_1)$ are interpolation couples and that the bilinear operator $T$ acts as a bounded transformation as indicated below:
\begin{align*}
                T: \mathcal A_0 \times \mathcal B_0 \rightarrow \mathcal C_0,\\
	        T: \mathcal A_1 \times \mathcal B_1 \rightarrow \mathcal C_1.
\end{align*}
If $\theta_0 \in (0,1)$, $p,q,r \in [1, \infty]$, and $1+ 1/r = 1/ p + 1/q$, then $T$ also acts as a bounded transformation in the following way:
\begin{align*}
                T: (\mathcal A_0, \mathcal A_1)_{\theta, p} \times (\mathcal B_0, \mathcal B_1)_{\theta, q} \rightarrow (\mathcal C_0, \mathcal C_1)_{\theta, r}.
\end{align*}
\end{lem}

\begin{lem}[See pp. 76-77 in \cite{BL}] \label{lem: interp T}
Suppose that $(\mathcal A_0, \mathcal A_1)$, $(\mathcal B_0, \mathcal B_1)$, and $(\mathcal C_0, \mathcal C_1)$ are interpolation couples and that the bilinear operator $T$ acts as a bounded transformation as indicated below:
\begin{align*}
                T: \mathcal A_0 \times \mathcal B_0 \rightarrow \mathcal C_0,\\
	        T: \mathcal A_0 \times \mathcal B_1 \rightarrow \mathcal C_1,\\
		T: \mathcal A_1 \times \mathcal B_0 \rightarrow \mathcal C_1.
\end{align*}
If $\theta_0, \theta_1 \in (0,1)$ and $p,q,r \in [1, \infty]$ are such that $1/ p+ 1/q \geq 1$, then $T$ also acts as a bounded transformation in the following way:
\begin{align*}
                T: (\mathcal A_0, \mathcal A_1)_{\theta_0, pr} \times (\mathcal B_0, \mathcal B_1)_{\theta_1, qr} \rightarrow (\mathcal C_0, \mathcal C_1)_{\theta_0+\theta_1, r}.
\end{align*}
\end{lem}
\section{Proof of \Str estimates for admissible exponents} \label{sect: proof 1}

In this paragraph we prove only the validity of the estimates in Theorem \ref{thm: Strich admiss KT}. The investigation of their sharpness shall be made in Section \ref{sect: counterex} by means of counterexamples.

Our plan is the following one. We  shall first prove the homogeneous estimate
\begin{align*}
  \normQRP{U(t)f}{q}{r}{r'} \lesssim& \normP{f}{2},
\end{align*}
via the corresponding estimate for the $TT^*$-operator for non-endpoint exponent triplets. Then in view of the invariance \eqref{power transform} we obtain
\begin{align*}
  \normQRP{U(t)f}{q}{r}{p} \lesssim& \normP{f}{a}, \quad 1\leq a \leq \infty, \; q>a.
\end{align*}
By duality and composition this implies the estimate
\begin{align*}
  \normQRP{W(t)F}{q}{r}{p} \lesssim&  \normQRP{F}{\tilde q'}{\tilde r'}{\tilde p'}
\end{align*}
whenever $a=\HM(p,r)=\HM(\tilde p', \tilde r')$ and $\tript$ is another non-endpoint KT-admissible exponent triplet.

\begin{proof}
In view of the decay estimate
\begin{align*}
    \normBC{U(t)f}{r}{r'} \lesssim \rp {\abs{t}^{\beta(r)}} \normBC{f}{r'}{r}, \quad 2\leq r \leq \infty,
\end{align*}
where $\beta(r) = n(1-2/r)$, we have
\begin{align*}
    \normBC{TT^*F}{r}{r'} \lesssim \Int \normBC{U(t-s)F(s)}{r}{r'} ds \lesssim \Int \frac {\normBC{F(s)}{r'}{r}} {\abs{t-s}^{\beta(r)}}  ds.
\end{align*}
We take the $L^q$-norm in $t$ and in view of the Hardy-Littlewood-Sobolev (HLS) theorem of fractional integration, see \cite[pp. 228-229]{BS}, \cite{S}, we obtain
\begin{align*}
    \normQRP{TT^*F}{q}{r}{r'} \lesssim \normQRP{F}{q'}{r'}{r},
\end{align*}
whenever $0<\beta(r)<1$, $1 + 1/q = 1/q' + \beta(r)$.  The latter conditions are equivalent to $2 < r < r^*(2)$, $1/q+n/r=n/2$. The left endpoint $r=2$ follows trivially from the transport estimate (\ref{eq: tran est 2}).
\end{proof}

The right endpoint $r = r^*(2)$ remains unresolved in the context of the KT equation, unlike that of the wave and the \Sch equations, where it has been resolved (in the positive) by Keel and Tao \cite{KT} (1997). In the setting of the inhomogeneous estimates, the ``double endpoint'' for which both exponent triplets are endpoint remains unresolved. However, if $\trip$ is non-endpoint and $\tript$ is endpoint the corresponding inhomogeneous estimate is still non-endpoint (since in such case $q > \tilde q'$) and thus holds true in view of Theorem \ref{thm: glob inhom est KT} whose proof will be given in Section \ref{sect: proof glob inhom}.

\section{Local inhomogeneous estimates} \label{sect: loc inhom est KT}

In order to go beyond the ``standard'' \Str estimates for the KT equation proved in the previous section we shall adapt and apply techniques pioneered by Foschi \cite{F}, and Keel and Tao \cite{KT}. We have also considered the works by Vilela \cite{V} and Taggart \cite{Tg}.

Our goal is to find the maximal range of estimates where we have the scaling property
\begin{align} \label{eq: local scal KT}
 \normLL{W(t)[\chi_{\lambda J}F]} \lesssim  \lambda^{\bbeta}\normRL{F}, \quad \forall \lambda >0,
\end{align}
for any two unit intervals $I$ and $J$ separated by a unit distance and any $F \in \SpaceLqrptdloc{q}{r}{p}{\R}$, where $\chi_{\lambda J}$ denotes the characteristic function of the rescaled interval $\lambda J$. Note that \eqref{eq: local scal KT} is equivalent to \eqref{BilStrichQ}.

\begin{lem}\label{lem: local estim 1 KT}
Estimate \eqref{eq: local scal KT}  holds for any two non-endpoint KT-admissible triplets $\trip$ and $\tript$ with $a$ $=$ $\tilde a'$. 
\end{lem}
\begin{proof} The proof follows trivially from Theorem \ref{thm: Strich admiss KT} due to the fact that $\beta(q,r,\tilde q, \tilde r)$ $=$ $0$ under the hypothesis of the lemma.
\end{proof}

\begin{lem}\label{lem: local estim 2 KT}
Estimate \eqref{eq: local scal KT}  holds with $\trip=(\infty, r, p)$ and $\tript=(\infty, p', r')$, where $1 \leq p \leq r \leq \infty$.
\end{lem}
\begin{proof}
Due to the decay estimate (\ref{DispEst}) we have that
\begin{align*}
    \sup_{t \in \lambda I} \norm{W(t)[\chi_{\lambda J}F]}_{ L^{r}_xL^{p}_v} &\lesssim \sup_{t \in \lambda I}\int_{\lambda J} \frac                    {\norm{F(\tau)}_{ L^{p}_xL^{r}_v}}             {\abs{t-\tau}^\betta}d\tau \\
    &\lesssim  \lambda^{\beta(\infty, r, \infty, p')}\norm{F}_{L^1(\lambda J; L^{p}_xL^{r}_v)}. \qedhere
\end{align*}
\end{proof}

\begin{lem}\label{lem: local estim 3 KT}
Whenever $\trip$ and $\tript$ are exponent triplets for which estimate \eqref{eq: local scal KT} holds, we have that \eqref{eq: local scal KT} also holds with $(Q,r,p)$ and $(\tilde Q, \tilde r, \tilde p)$, where $1\leq Q \leq q$, $1 \leq \tilde{Q} \leq \tilde{q}$.
\end{lem}
\begin{proof} A trivial application of H\"older's inequality
\begin{equation*}
 \begin{split}
  \norm{W(t)[\chi_{\lambda J}F]}_{L^{Q}(\lambda I; L^{r}_xL^{p}_v)} \lesssim
  \lambda^{\rp Q - \rp q} \norm{W(t)[\chi_{\lambda J}F]}_{L^{q}(\lambda I; L^{r}_xL^{p}_v)} \qquad\\
  \lesssim \lambda^{\beta(Q,r,\tilde q , \tilde r)}\norm{F}_{L^{\tilde{q}'}(\lambda J; L^{\tilde{r}'}_xL^{\tilde{p}'}_v)}
  \lesssim  \lambda^{\beta(Q,r,\tilde Q , \tilde r)}\norm{F}_{L^{\tilde{Q}'}(\lambda J; L^{\tilde{r}'}_xL^{\tilde{p}'}_v)}. \qedhere
 \end{split}
\end{equation*}
\end{proof}

Let us define the range of validity of the local estimates \eqref{eq: local scal KT} as the set $\mathcal{E} \subset \R^6$ consisting of exponent vectors $\triptrip$ in $\R^6$ that correspond to valid estimates \eqref{eq: local scal KT}. We shall only find the convex hull $\mathcal{E}_0$ $\subseteq \mathcal{E}$ of the points in $\R^6$ that correspond to the estimates in the three lemmas above. The question whether $\mathcal{E}_0$ $= \mathcal{E}$ remains open. 

\begin{lem}[Local inhomogeneous estimates]\label{lem: local inhom estim KT}
Estimate \eqref{eq: local scal KT}  holds whenever the exponent triplets $\trip$, $\tript$ satisfy the following conditions
\begin{align}
 0 \leq \rp q, \rp {\tilde{q}} \leq 1, & \quad 0 < \rp p, \rp {\tilde{p}}, \rp r, \rp {\tilde{r}} \leq 1,  \\
\rp r \leq \rp p, \quad \rp {\tilde r} \leq \rp {\tilde p},& \quad \HM(p,r)=\HM(\tilde p', \tilde r'),\\
\qconds, &\quad \qcondts, \qquad \label{cond: loc inhom 2 KT}\\
\frac {n -1}{p'} < \frac {n}{\tilde{r}}, &\quad \frac {n-1}{\tilde{p}'} < \frac {n}{r},\qquad\qquad\qquad \label{cond: loc inhom <} 
\end{align}
or if the point $(1/q, 1/r, 1/p, 1/\tilde q, 1/ \tilde r, 1/ \tilde p)$ lies inside one of the ``cubic'' sets in $\R^6$ below
\begin{equation} \label{omitted points in lemma}
  \begin{split}
   \left(\kappa,0,\mu,\nu,1 - \mu,1\right), \qquad 0 \leq \kappa, \mu, \nu \leq 1, \qquad\\
   \left(\kappa,1-\mu,1,\nu,0,\mu \right),  \qquad 0 \leq \kappa, \mu, \nu \leq 1.\qquad
  \end{split}
\end{equation}
\end{lem}
\begin{proof}
We apply the Riesz-Thorin convexity theorem to interpolate between the already proven local estimates. To that end we need to find the convex hull of the sets in $\R^6$ associated with Lemmas \ref{lem: local estim 1 KT} and \ref{lem: local estim 2 KT} and then expand that set by the rule given in Lemma \ref{lem: local estim 3 KT}.

The range of validity $S_1$ of the local estimates in Lemma \ref{lem: local estim 1 KT} is given by the system
\begin{gather}
 0 < \rp r, \rp {\tilde{r}} \leq 1,  \quad 0 \leq \rp q, \rp {\tilde{q}}, \rp p, \rp {\tilde{p}}\leq 1,  \label{sys: 1 line 1}\\
  \frac 1 q = \frac n 2 \left(\rp p - \rp r \right), \quad \frac 1 {\tilde q} = \frac n 2 \left(\rp {\tilde p} - \rp {\tilde r} \right), \label{sys: 1 line 2}\\
  \rp r + \rp p + \rp {\tilde r} + \rp {\tilde p} = 2, \label{sys: 1 line 3}\\
  \frac {n-1}{p} < \frac {n+1}{r}, \qquad \frac {n-1}{\tilde p}  < \frac {n+1}{\tilde r},\label{sys: 1 line 4}
\end{gather}
or if $(1/q, 1/r, 1/p, 1/\tilde q, 1/ \tilde r, 1/ \tilde p) \in \{B=(0,0,0,0,1,1),  C=(0,1,1,0,0,0)\}$.

Note that $S_1$ is a convex polyhedron in $\R^6$ and the two points $B$ and $C$ lie on its boundary. The range of validity $S_2$ of the local estimates in Lemma \ref{lem: local estim 2 KT} is the convex hull, in fact a triangle, of the three points
\begin{equation} \label{triangle}
 A=(0,0,1,0,0,1), \qquad B=(0,0,0,0,1,1), \qquad C=(0,1,1,0,0,0).
\end{equation}
The vertices $B$ and $C$ are already included in $S_1$ and thus it would suffice to take only the vertex $A$. Hence, we obtain the following set
\begin{align*}
 \rp Q = \frac \theta q, \qquad \rp R = \frac \theta r, \qquad \rp P = 1-\theta + \frac \theta p,\qquad\qquad\qquad\;\;\;\\
 \rp {\tilde Q} = \frac \theta {\tilde q}, \qquad \rp {\tilde R} = \frac \theta {\tilde r}, \qquad \rp {\tilde P} = 1-\theta + \frac \theta {\tilde p}, \qquad 0 \leq \theta \leq 1,
\end{align*}
where $\Triptrip$ are the coordinates of the new set $S_3$ written in terms of $\triptrip$ and $\theta$. Of course, we must also add to $S_3$ the line segments $[A,B]$ and $[A,C]$. We shall treat this case separately at the end.

Finally, we apply the rule given in Lemma \ref{lem: local estim 3 KT} and thus we replace the equations for $Q$ and $\tilde Q$ above with the following inequalities
\begin{align*}
 1 \geq \rp Q \geq \frac \theta q,\qquad 1 \geq \rp {\tilde Q} \geq \frac \theta {\tilde q},
\end{align*}
plus the restrictions
\begin{align} \label{sys: 1 line 5}
 \rp r \leq \rp p, \quad \rp {\tilde r} \leq \rp {\tilde p},
\end{align}
which were implicitly assumed in \eqref{sys: 1 line 2}.

{\bf 1.} We first eliminate $q$ and $\tilde q$ from the system for $S_1$ to obtain
\begin{align*}
 \rp Q \geq \frac n 2 \left( \frac \theta p - \frac \theta r \right), \quad &\Leftrightarrow \quad \rp Q \geq \frac n 2 \left( \theta - 1 + \rp P - \rp R \right),\\
 &\Leftrightarrow \quad \theta \leq \rp {P'} + \rp R + \frac 2 {nQ}.
\end{align*}
Similarly,
\begin{equation*}
 \qquad \quad\theta \leq \rp {\tilde P'} + \rp {\tilde R} + \frac 2 {n\tilde Q}, \qquad \rp Q, \rp {\tilde Q} \leq 1.\quad
\end{equation*}

{\bf 2.} As expected, condition (\ref{sys: 1 line 3}) is invariant
\begin{equation*}
 \scondAB.
\end{equation*}

{\bf 3.} Reworking condition (\ref{sys: 1 line 4}), we obtain
\[
\theta < \frac {n+1}{n-1} \rp R + \rp {P'}, \qquad \theta < \frac {n+1}{n-1} \rp {\tilde R} + \rp {\tilde P'}.
\]

{\bf 4.} Condition (\ref{sys: 1 line 5}) is replaced by
\[
\rp {P'} + \rp R \leq \theta, \quad \rp {\tilde P'} + \rp {\tilde R} \leq \theta.
\]

{\bf 5.} Finally, conditions (\ref{sys: 1 line 1}) are transformed into
\[
  \rp {P'}, \rp {\tilde{P}'}, \rp R, \rp {\tilde R} \leq \theta, \quad 0 \leq \rp Q, \rp {\tilde Q}, \rp P, \rp {\tilde{P}}, \rp R, \rp {\tilde{R}}\leq 1.
\]

{\bf 6.} We group all conditions obtained in the previous 5 steps according to their type
\begin{align}
  0,\, \rp {P'},\, \rp {\tilde{P}'},\, \rp R,\, \rp {\tilde{R}}, \, \rp {P'} + \rp R, \, \rp {\tilde P'} + \rp {\tilde R}  \leq \theta. \qquad\qquad\qquad\label{sys: 2 line 1}\\
   \theta \leq \rp R + \rp {P'} + \frac 2 {nQ},\,  \rp {\tilde{R}} + \rp {\tilde{P}'} + \frac 2 {n\tilde{Q}},
  \frac {n+1}{n-1} \rp R + \rp {P'},\, \frac {n+1}{n-1} \rp {\tilde{R}} + \rp {\tilde{P}'},\, 1. \label{sys: 2 line 2}\\
  0 \leq \rp Q,\, \rp {\tilde{Q}},\, \rp P,\, \rp {\tilde{P}},\, \rp R,\, \rp {\tilde{R}} \leq 1,\, \qquad \rp P + \rp R + \rp {\tilde{R}} + \rp {\tilde{P}} = 2.\qquad \label{sys: 2 line 3}
\end{align}

{\bf 7.} We discard the redundant conditions like
\[
 0,\, \rp {P'},\, \rp {\tilde{P}'}, \rp {R},\, \rp {\tilde R} \leq \theta,
\]
which are all weaker than the other two in \eqref{sys: 2 line 1}.

There exists $\theta$ solving all inequalities in \eqref{sys: 2 line 1}, \eqref{sys: 2 line 2}, if and only if every quantity in \eqref{sys: 2 line 1} is bounded from above by any quantity in \eqref{sys: 2 line 2}. Thus we form all possible combinations between the quantities in the two types of (reduced) inequalities  to obtain the following set of conditions
\begin{align*}
 \rp R + \rp {P'} \leq \rp {\tilde{R}} + \rp {\tilde{P}'} + \frac 2 {n\tilde{Q}}, &\quad \rp {\tilde{R}} + \rp {\tilde{P}'} \leq \rp R + \rp {P'} + \frac 2 {nQ},\\
 \rp R + \rp {P'} < \frac {n+1}{n-1} \rp {\tilde{R}} + \rp {\tilde{P}'},  &\quad \Leftrightarrow \quad \frac {n -1}{P'} < \frac {n}{\tilde{R}},\\ 
 \rp {\tilde{R}} + \rp {\tilde{P}'} <  \frac {n+1}{n-1} \rp R + \rp {P'},  &\quad \Leftrightarrow \quad  \frac {n-1}{\tilde{P}'} < \frac {n}{R},\\
 \rp R \leq \rp P, &\quad \rp {\tilde R} \leq \rp {\tilde P},
\end{align*}
describing the region $S_3$.

{\bf 8.} We apply the rule given in Lemma \ref{lem: local estim 3 KT} to the two line segments $[A,B]$ and $[A, C]$ to obtain the following two ``cubic'' regions in $\R^6$
\begin{equation} \label{omitted points}
  \begin{split}
 \left(\mu,0,\kappa,\nu,1 - \kappa,1\right), \qquad 0 \leq \mu, \nu, \kappa \leq 1,\qquad\\
 \left(\mu,1-\kappa,1,\nu,0,\kappa \right),  \qquad 0 \leq \mu, \nu, \kappa \leq 1.\qquad
  \end{split}
\end{equation}
Hence, the computation of the set $\mathcal{E}_0$ is finished.
\end{proof}

\section{Proof of the generalized inhomogeneous Strichartz estimates} \label{sect: proof glob inhom}

\subsection{Generalized inhomogeneous non-endpoint estimates} \label{sect: glob inhom non-endp}
In this paragraph we prove the inhomogeneous \Str estimate
\[
 \normQRP{W(t)F}{q}{r}{p} \lesssim \normQRPD{F}{q}{r}{p}, \quad \forall F \in \spaceLqrptd,
\]
in the range $q > \tilde q'$. Thanks to Lemma \ref{lem: bil A}, we have reduced this problem to showing the estimate
\[
 \norm{\{ b_\lambda\}}_{l^1} \lesssim \normQRP{F}{\tilde q'}{\tilde r'}{\tilde p'}\normQRP{G}{q'}{r'}{p'}, \quad \forall F \in \spaceLqrptd,\; \forall G \in \spaceLqrpd,
\]
where
\begin{align*}
 \left\{ b_\lambda \right\}_{\lambda \in 2^{\Z}} = \left\{ \sum_{Q \in \Ol} \abs{B_Q(F, G)} \right\}_{\lambda \in 2^{\Z}}.
\end{align*}
We shall next specify the range of validity of the above estimates in terms of the vector
\[
  P=\triptrip \in \mathcal{E}_0
\]
for which
\begin{align}
  1/q + 1/ {\tilde q} = n&\left(1-1/ r - 1/ {\tilde r}\right). \label{cond: scal 1/q}
\end{align}
Let us denote by $\Delta$ the set in $\R^2$
\begin{align*}
  \{(x,y) | x>0, y > 0, x+y < 1, (x, 1/r, 1/p, y, 1/\tilde r, 1/ \tilde p) \in \mathcal{E}_0 \}
\end{align*}
and its interior (largest open subset) by int($\Delta$). In this paragraph we prove that for any such vector $P \in$ int($\Delta$) the corresponding inhomogeneous \Str estimate holds true.

Under the assumption of the latter condition and in view of Corollary \ref{cor: 1/q+} we have the estimate
\[
 \abs{b_\lambda} \lesssim \lambda^{\beta(q,r,\tilde q, \tilde r)} \normQRP{F}{\tilde q'}{\tilde r'}{\tilde p'}\normQRP{G}{q'}{r'}{p'},
\]
or equivalently, $\{ b_\lambda\} \in l^\infty_{s}$ with $s=-\beta(q,r,\tilde q, \tilde r)$. Let us set
\[
 1/q_0 = 1/q + \epsilon, \quad  1/\tilde q_0 = 1/\tilde q + \epsilon, \quad 1/q_1 = 1/q - 3\epsilon, \quad  1/\tilde q_1 = 1/\tilde q - 3\epsilon,
\]
for some small enough $\epsilon >0$, whose existence is guaranteed by our assumptions, such that the perturbed exponent vectors do not leave int($\Delta$). Then we have that $\beta(q_0, r, \tilde q_0, \tilde r)$ $=$ $2\epsilon$,  and $\beta(q_1, r, \tilde q_0, \tilde r)$ $=$ $\beta(q_0, r, \tilde q_1, \tilde r)$ $=$ $-2\epsilon$. The following bilinear maps
\begin{align*}
                A: \SpaceLqrptd{q_0}{r}{p}  \times \SpaceLqrpd{q_0}{r}{p} &\rightarrow l_{-2\epsilon}^\infty,\\
	        A: \SpaceLqrptd{q_0}{r}{p}  \times \SpaceLqrpd{q_1}{r}{p} &\rightarrow l_{2\epsilon}^\infty,\\
		A: \SpaceLqrptd{q_1}{r}{p}  \times \SpaceLqrpd{q_0}{r}{p} &\rightarrow l_{2\epsilon}^\infty,
\end{align*}
are bounded. In virtue of Lemma \ref{lem: interp T}, we have that the map
\begin{align*}
                A: (\SpaceLqrptd{q_0}{r}{p}, \SpaceLqrptd{q_1}{r}{p})_{1/4, \tilde q'}  \times (\SpaceLqrpd{q_0}{r}{p}, \SpaceLqrpd{q_1}{r}{p})_{1/4, q'}   \rightarrow (l_{2\epsilon}^\infty, l_{-2\epsilon}^\infty)_{1/2,1}
\end{align*}
is also bounded. Finally, in view of Lemma \ref{lem: interp l} and the embeddings of the Lorentz spaces, we obtain
\begin{align*}
                A: \spaceLqrptd  \times \spaceLqrpd \rightarrow l^1.
\end{align*}

All assumption made in this paragraph are explicitly stated in Theorem \ref{thm: glob inhom est KT}, part (i). We remark that condition \eqref{cond: loc inhom 2 KT} together with \eqref{cond: scal 1/q} is equivalent to $\trip$ and $\tript$ being KT-acceptable. Furthermore, in this case the inequalities in \eqref{cond: loc inhom 2 KT} have to be taken as strict inequalities so that $P \in$ int($\Delta$). Let us also note that the two locally acceptable ``cubic'' sets in \eqref{omitted points} give rise to the two globally acceptable ``cubic cross section'' sets $\Sigma_1$ and $\Sigma_2$ in Definition \ref{dfn: joint accept}.

\subsection{Global inhomogeneous endpoint estimates with $\tilde q=\infty$} \label{sect: glob inhom infty}

In this paragraph we prove the inhomogeneous \Str estimates with $P$ lying on either of the two catheti of $\Delta OAB$ in Figure \ref{fig: 1}. Since by duality both type of estimates are equivalent, it is enough to consider only the case $\tilde q=\infty$. We exclude the two endpoints $(0,0)$ and $(1,0)$ from our considerations. We suppose that
\[
  P(1/q, 1/r, 1/p, 0, 1/\tilde r, 1/ \tilde p) \in \{0<1/q<1\} \cap \mathcal{E}_0
\]
and is such that $q$ satisfies every inequality of $\mathcal E_0$ as a strict inequality. We also assume the scaling condition \eqref{cond: scal 1/q} from the previous paragraph. Then we have
\begin{align*}
  A: \SpaceLqrptdMod{1}{r}{p}  \times \SpaceLqrpd{q_0}{r}{p} &\rightarrow l_{\epsilon}^\infty,\\
  A: \SpaceLqrptdMod{1}{r}{p}  \times \SpaceLqrpd{q_1}{r}{p} &\rightarrow l_{-\epsilon}^\infty,
\end{align*}
where
\begin{align*}
 \rp {q_0} = \rp q - \rp \epsilon, \quad \rp {q_1} = \rp q + \rp \epsilon.
\end{align*}
The real method with parameters $(\theta, q)$ = $(1/2, 1)$, see Lemma \ref{lem: interp T1}, gives that
\begin{align*}
  A: \SpaceLqrptdMod{1}{r}{p}  \times \SpaceLqrpdMod{q}{1}{r}{p} &\rightarrow l^1.
\end{align*}
Equivalently, in view of the $TT^*$-principle,
\begin{align}\label{eq: inhom pertb}
 \normQRP{W(t)F}{q,\infty}{r}{p} \lesssim \normQRP{F}{1}{\tilde r'}{\tilde p'},
\end{align}
for all $F \in \SpaceLqrptdMod{1}{r}{p}$.
The explicit restrictions on the Lebesgue exponents $\trip$ and $(\infty, \tilde r, \tilde p)$ are stated in Theorem \ref{thm: glob inhom est KT}, part (ii). Analogously, the dual case is stated as part (iii) of that theorem.

In the remainder of this paragraph we address (i) the corresponding homogeneous estimates to \eqref{eq: inhom pertb} via the Equivalence Theorem \ref{thm: equiv KT} in its stronger form for Lorentz spaces and (ii) the sharpening of \eqref{eq: inhom pertb} to the Lebesgue norm $L^q_t$.

\begin{lem} \label{lem: hom bc}
The estimate
\begin{align} \label{eq: hom bc inf}
 \normQRP{U(t)f}{q, \infty}{r}{p} \lesssim \normBC{f}{b}{c},
\end{align}
holds for all $f \in \spaceLbc$, whenever
\begin{align*}
 &\rp q + \frac n r = \frac n b, \quad \HM(r,p) = \HM(b, c) \stackrel{def}{=} a,  \quad r < \frac {nc} {n-1},\\
 &p < b \leq a \leq c < r, \quad 1 < q < \infty,\, 1 \leq p, \tilde p, r, \tilde r < \infty.
\end{align*}
\end{lem}
\begin{proof}
The range of validity of estimate \eqref{eq: hom bc inf} is determined in the following way. We first write the conditions defining the set $\mathcal E_0$, however, any inequality where $q$ appears is taken as a strict inequality. To that system we add the scaling condition \eqref{cond: scal 1/q}. Thus, we have that $1/q$ $=$ $n\left(1-1/ r - 1/ {\tilde r}\right)$, and
\begin{gather*}
 0 < \rp q, \rp {\tilde{q}} < 1,  \quad 0 < \rp p, \rp {\tilde{p}}, \rp r, \rp {\tilde{r}} \leq 1,  \\
 \rp r \leq \rp p, \quad \rp {\tilde r} \leq \rp {\tilde p},  \quad \HM(p,r)=\HM(\tilde p', \tilde r'),\\
 \rp q < n \left( \rp p - \rp r \right), \quad 0 \leq n \left( \rp {\tilde p} - \rp {\tilde r} \right), \\
\frac {n -1}{p'} < \frac {n}{\tilde{r}}, \quad \frac {n-1}{\tilde{p}'} < \frac {n}{r},
 \end{gather*}
or that the point $(1/q, 1/r, 1/p, 0, 1/ \tilde r, 1/ \tilde p)$ belongs to the set
\begin{equation*}
   \left(\kappa, 0, \mu,   0, 1-\mu, 1\right), \quad 0 < \kappa, \mu < 1, \; \kappa=n\mu.
\end{equation*}
The latter set of exponents does not give us anything new as it essentially expresses a special case of the decay estimate
\begin{align*}
 \normQRP{U(t)f}{q, \infty}{\infty}{nq} \lesssim \normBC{f}{nq}{\infty}.
\end{align*}
Let us use the more natural notation for the exponents $b$ $=$ $\tilde r'$ and $c$ $=$ $\tilde p'$. Thus we get the following system of conditions
\begin{align*}
\rp q + \frac n r = \frac n b, \quad &\HM(p,r)=\HM(\tilde p', \tilde r'),\\
r < \frac {nc}{n-1},           \quad &p < b \leq c < r, \\
1 <q < \infty,       \quad &1 \leq p, \tilde p, r, \tilde r < \infty. \qedhere
\end{align*}
\end{proof}

\begin{cor} \label{cor: hom bc}
 The estimate
\begin{align} \label{eq: hom bc}
 \normQRP{U(t)f}{q}{r}{p} \lesssim \normBC{f}{b}{c}, \quad b\neq c,
\end{align}
holds for all $f \in \spaceLbc$ whenever
\begin{align*}
 \rp q + \frac n r = \frac n b, \quad p < b < a < c < r, \quad r < \frac n {n-1} c, \quad q \geq c,\\
 \HM(r,p) = \HM(b, c) \stackrel{def}{=} a, \quad 1 < q, b, c, r < \infty, \quad 1 \leq p < \infty.
\end{align*}
The $L^q_t$-norm in \eqref{eq: hom bc} can be replaced by the $L^{q,c}_t$-norm. In such case the assumption $q\geq c$ can be removed.
\end{cor}
\begin{proof}
Each estimate in the statement of this corollary can be proved by interpolating two estimates \eqref{eq: hom bc inf} with the real method. Indeed, let us perturb slightly the exponents $q$, $b$, and $c$, keeping $r$ and $p$ fixed, in such a way that they remain in the range of validity of the estimates \eqref{eq: hom bc inf}. For example, the perturbed exponents can be taken as follows
\begin{align*}
 1/q_1 = 1/q + n/\epsilon, \quad 1/b_1 = 1/b + 1/\epsilon, \quad 1/c_1 = 1/c - n/\epsilon, \\
 1/q_2 = 1/q - n/\epsilon, \quad 1/b_2 = 1/b - 1/\epsilon, \quad 1/c_2 = 1/c + n/\epsilon.
\end{align*}
We then interpolate by the real method with $(\theta, q)$ = $(1/2, c)$, and make use of Lemma \ref{prop: real interp 1}.
\end{proof}

Let us remark that the case of $b=c$, excluded in Corollary \ref{cor: hom bc}, is not new and is considered in Theorem \ref{thm: Strich admiss KT}.

\begin{lem} \label{lem: inhom restr infty}
Suppose that  $\trip$ and $(\infty, \tilde r, \tilde p)$ are two jointly KT-acceptable exponent triplets and $1$ $<$ $\tilde p'$ $\leq$ $q$ $<$ $\infty$. Then the estimate
\begin{align*}
\normQRP{W(t)F}{q}{r}{p} \lesssim \normQRP{F}{1}{\tilde r'}{\tilde p'}, 
\end{align*}
holds for all $F \in \SpaceLqrp{1}{\tilde r'}{\tilde p'}$. Similarly, if $(\infty, r, p)$ and $\tript$ are  two jointly KT-acceptable exponent triplets and $1$ $<$ $\tilde q'$ $\leq p$ $<$ $\infty$, then the estimate
\begin{align*}
\normQRP{W(t)F}{\infty}{r}{p} \lesssim \normQRP{F}{\tilde q'}{\tilde r'}{\tilde p'} 
\end{align*}
holds for all $F \in \SpaceLqrp{\tilde q'}{\tilde r'}{\tilde p'}$.
\end{lem}
\begin{proof}
 The lemma follows directly from Lemma \ref{lem: hom bc}, Corollary \ref{cor: hom bc}, and the Equivalence Theorem \ref{thm: equiv KT}. The range of validity of these estimates is identical to that of the generalized homogeneous estimates except for the usual change of notation.

Let us verify that the range of the exponents is the same as that assumed in Theorem \ref{thm: glob inhom est KT}. The assumption there is that $\trip$ and $(\infty, \tilde r, \tilde p)$ are jointly KT-acceptable and that $1<q<\infty$. This immediately implies the following range
\[
 1 < q < \infty, \quad 1 \leq p, r, \tilde p, \tilde r < \infty. 
\]
Next, the requirement that $\trip$ is KT-acceptable and that $q<\infty$ leads to $p<r$. Therefore $r>1$. The scaling condition \eqref{cond: scal 1/q} together with the fact that $\trip$ is KT-acceptable implies that $p<\tilde r'$. The identity $\HM(p,r)=\HM(\tilde p, \tilde r)$ together with $p<r$, and $\tilde p \leq \tilde r$, and $p<\tilde r'$, leads to
\[
  p < \tilde r' \leq \tilde p' < r.
\]
The latter implies that $1< \tilde p \leq \tilde r < \infty$. Thus we obtain the following range
\[
 1 < q, r, \tilde p, \tilde r < \infty, \quad 1 \leq p < \infty, 
\]
which is the range for which the estimates in this lemma are proven. Analogously for the dual case.
\end{proof}

\subsection{Global inhomogeneous endpoint estimates with $q=\tilde q'$} \label{sect: endp glob inhom hypo KT}

In this paragraph we assume that the $L^p_v$-norms are given over a bounded velocity space $V \subset \R^n$ and prove the inhomogeneous estimates \eqref{eq: Strich inhom endp}.

We suppose now that $P$  lies on the hypotenuse of $\Delta OAB$ in Figure \ref{fig: 1} and that it also belongs to $\mathcal E_0$. The 4-vector $(1/r,1/p, 1/\tilde r, 1 /\tilde p)$ should satisfy every inequality in $\mathcal E_0$ as a strict inequality. Of course, we cannot remove the restriction $\HM(p,r)=\HM(\tilde p',\tilde r')$, but we shall perturb these exponents in such a way that they always satisfy the latter condition. The exponents $(1/q, 1/\tilde q)$ will remain fixed throughout this paragraph. We consider the following perturbations

\begin{align*}
 \rp r_0 = \rp r + \epsilon, \quad\;\;  \rp {\tilde r_0}= \rp {\tilde r} + \epsilon, \quad\;\; \rp p_0 = \rp p - \epsilon, \quad\;\;  \rp {\tilde p_0} = \rp {\tilde p} - \epsilon,\;\;\\
\rp r_1 = \rp r - 3\epsilon, \quad  \rp {\tilde r_1} = \rp {\tilde r} - 3\epsilon, \quad \rp p_1 = \rp p + 3\epsilon, \quad  \rp {\tilde p_1} = \rp {\tilde p} + 3\epsilon.
\end{align*}

We have that $\beta(q, r_0, \tilde q, \tilde r_0) = 2n\epsilon$  and $\beta(q, {r}_1, \tilde q, \tilde {r}_0)=\beta(q, {r}_0, \tilde q, \tilde {r}_1)  = -2n\epsilon$. Hence the maps
\begin{align*}
                A: \SpaceLqrpd{\tilde q}{\tilde r_0}{\tilde p_0}  \times \SpaceLqrpd{q}{r_0}{p_0} &\rightarrow l_{-2\epsilon}^\infty,\\
	        A: \SpaceLqrpd{\tilde q}{\tilde r_0}{\tilde p_0}  \times \SpaceLqrpd{q}{r_1}{p_1} &\rightarrow l_{2\epsilon}^\infty,\\
		A: \SpaceLqrpd{\tilde q}{\tilde r_1}{\tilde p_1}  \times \SpaceLqrpd{q}{r_0}{p_0} &\rightarrow l_{2\epsilon}^\infty,
\end{align*}
are bounded.  In virtue of Lemma \ref{lem: interp T} and the well-known interpolation identity
\begin{align} \label{eq: Lp interp}
 (L^p(\R; \mathcal A_0), L^p(\R; \mathcal A_1))_{\theta, p} = L^p(\R; (\mathcal A_0, \mathcal A_1)_{\theta, p}), \quad 1<p<\infty,
\end{align}
see \cite{BL}, the map
\begin{align*}
         A: (\SpaceLqrpd{\tilde q}{\tilde r_0}{\tilde p_0},& \SpaceLqrpd{\tilde q}{\tilde r_1}{\tilde p_1})_{1/4, \tilde q'}  \times \\ &(\SpaceLqrpd{q}{r_0}{p_0}, \SpaceLqrpd{q}{r_1}{p_1})_{1/4, q'}  \To (l_{2\epsilon}^\infty, l_{-2\epsilon}^\infty)_{1/2,1}
\end{align*}
is also bounded. In view of the fact that $V$ is bounded we have that $L^{\tilde P'}(V) \hookrightarrow L^{\tilde p_0'}(V)$ and $L^{\tilde P}(V) \hookrightarrow L^{\tilde p_1'}(V)$ whenever $1 \leq \tilde P \leq \min(\tilde p_0, \tilde p_1)$. Analogously,
$L^{P'}(V) \hookrightarrow L^{p_0'}(V)$ and $L^{P'}(V) \hookrightarrow L^{p_1'}(V)$ whenever $1 \leq P \leq \min(p_0, p_1)$. Thus we also have that the map
\begin{align*}
         A: (\SpaceLqrpd{\tilde q}{\tilde r_0}{\tilde P},& \SpaceLqrpd{\tilde q}{\tilde r_1}{\tilde P})_{1/4, \tilde q'}  \times \\ &(\SpaceLqrpd{q}{r_0}{P}, \SpaceLqrpd{q}{r_1}{P})_{1/4, q'}  \To (l_{2\epsilon}^\infty, l_{-2\epsilon}^\infty)_{1/2,1}
\end{align*}
is bounded. Finally, in view of the interpolation identity \eqref{eq: Lp interp}, it follows that
\begin{align*}
           A: \SpaceLqrp{\tilde q'}{\tilde r', \tilde q'}{\tilde P'} \times \SpaceLqrp{q'}{r', q'}{P'} \rightarrow l^1.
\end{align*}
In view of Lemma \ref{lem: bil A}, this implies the estimate
\begin{align}
\normQRPf{W(t)F}{q}{r,q}{P} \lesssim_V \normQRPf{F}{\tilde q'}{\tilde r', \tilde q'}{\tilde P'},
\end{align}
for any $P, \tilde P$, such that $1\leq P < p$ and $1 \leq \tilde P < \tilde p$, and any two jointly KT-acceptable exponent triplets $\trip$ and $\tript$ whose exponents further satisfy the following conditions $1 <q$, $\tilde q < \infty$, $q=\tilde q'$.

\section{Counterexamples} \label{sect: counterex}

In this section we give necessary conditions for the range of validity of the \Str estimates for the KT equation by means of counterexamples.

We first make the general remark that the validity of \Str estimates with exponents $r=\infty$ in the homogeneous setting, and with $r=\infty$ or $\tilde r=\infty$ in the inhomogeneous setting, is completely solved in \cite{Ovch4}. There it is proved that the only valid estimate of the form
\[
 \normQRP{U(t)f}{q}{\infty}{p} \lesssim \normBC{f}{b}{c}, \quad \forall f \in \spaceLbc
\]
in any spatial dimension is for $q=p=b=c=\infty$. Also, the only valid inhomogeneous estimates of the form
\[
 \normQRP{W(t)F}{q}{r}{p} \lesssim \normQRPtd{F}{q}{r}{p}, \quad \forall F \in \spaceLqrptd,
\]
in any spatial dimension with either $r=\infty$ or $\tilde r=\infty$ are only those whose exponents are explicitly stated in Definition \ref{dfn: joint accept}.

\subsection{Homogeneous estimates} By scaling, that is Lemma \ref{lem: scaling}, estimate
\begin{equation*}
 \normQRP{U(t)f}{q}{r}{p} \lesssim \normA{f}{a}, \quad \forall f \in \spaceLa{a}, 
\end{equation*}
 holds only if
\begin{align*}
 \rp q + \frac n r = \frac n a, \quad a=\HM(p,r).
\end{align*}

Let us next find the upper bound $r \leq r^*(a)$. It is enough to consider only the special case $a=2$. We shall prove the equivalent condition $q \geq 2$. (In general $r \leq r^*(a)$ and $q \geq a$ are equivalent.) The claim follows directly by the translation invariance in $t$ of the $TT^*$-operator. Indeed, first recall that the above estimate with $a=2$ is equivalent to
\[
 \normQRP{TT^*F}{q}{r}{r'} \lesssim \normQRP{F}{q'}{r'}{r}, \quad \forall F \in {L^{q'}_tL^{r}_xL^{r'}_v}.
\]
Then, in view of the famous H\"ormander's lemma \ref{lem: Horm}, we have that $q \geq q'$, or equivalently $q \geq 2$.

For finding a lower bound on $r$ we use the translation invariance in $x$ of $TT^*$ and thus we get that $r \geq r'$, or equivalently, $r \geq 2$. 
As usual, the condition $r\geq a$ in the general case $0<a<\infty$ follows by the power invariance (\ref{power transform}).

Let us verify the translation invariance in $t$ of $TT^*$. Consider $F_{\tau}(t) = F(t-\tau)$. For $TT^*F_{\tau}$ we have
\begin{align*}
    \Int U(t-s)F(s-\tau) ds = \Int U(t-\tau-\sigma)F(\sigma) d\sigma,
 \end{align*}
or in other words $\{TT^*F_{\tau}\}(t)$ $=$ $\{TT^*F\}(t-\tau)$.

\begin{lem}[H\"ormander \cite{H}] \label{lem: Horm}
 Whenever a (non-trivial) linear and bounded operator maps $L^p(\R^n)$ to $L^q(\R^n)$, $1\leq p,q < \infty$, and additionally this operator is translation invariant, then we must have that $p \leq q$.
\end{lem}

\begin{rem}
H\"ormander's lemma remains true in a more general setting. For example, the space $L^p$ and $L^q$ can be vector-valued, i.e.
$L^p(X;\mathcal B_1)$ and $L^q(X;\mathcal B_1)$ respectively, where $X \subseteq \R^n$ is the set $\{x=(x_1,\dots,x_n)| a_i < x_i < \infty, i=1,\dots n\}$ for some fixed $a_i$ $\in$ $\R\cup\{-\infty\}$, and $\mathcal B_1$ and $\mathcal B_2$ are some Banach spaces. Furthermore, the spaces $L^p$ and $L^q$ may be mixed Lebesgue spaces (or Bochner spaces in the vector-valued setting). Suppose for example that $p=(p_1, \dots, p_k)$ and $q=(q_1, \dots, q_l)$ and $L^p$ and $L^q$ are the corresponding mixed Lebesgue spaces with the usual notation. Consider the bounded linear operator $T: L^p \To L^q$. Let $u(x_1,\dots, x_k) \in L^p$, $v(y_1,\dots, y_l) \in L^q$, $\tau_h$ be the operator defined by
\[
 \tau_h u(x_1,\dots, x_i, \dots, x_k) = u(x_1,\cdots, x_i+h, \dots, x_k),
\]
and similarly let $\sigma_h$ be the operator defined by
\[
 \sigma_h v(y_1,\dots, y_j, \dots, y_l) = v(y_1,\dots, y_j+h, \dots, y_l).
\]
Then, if we have that
\[
 T\tau_h u = \sigma_h Tu, \quad \forall h \geq 0,
\]
it follows that either $q_j \geq p_i$, or $T=0$. The proof of that statement is virtually the same as that of Lemma \ref{lem: Horm}.
\end{rem}

\subsection{Generalized homogeneous estimates}  \label{sect: counterex mix}
Let us consider the homogeneous \Str estimate
\begin{equation*}
 \normQRP{U(t)f}{q}{r}{p} \lesssim \normBC{f}{b}{c},  \quad \forall f \in \spaceLbc, 
\end{equation*}
for data outside the transport class. Most of the arguments from the preceding paragraph apply to this case as well. By scaling, we have that the conditions
\begin{align} \label{cond: range r b}
 \rp q + \frac n r = \frac n b, \quad \HM(p,r)=\HM(b,c) \stackrel{def}{=}a,
\end{align}
are necessary. The following conditions
\begin{align}\label{cond: range r}
  r \geq p, \quad \rp q < n \left(\rp p - \rp r\right), \quad \text{or } \; q=\infty,\; 1 \leq\, p=r \,\leq \infty,
\end{align}
are also necessary. To that end, let us consider the equivalent estimate
\begin{equation*}
 \normQRP{TT^*F}{q}{r}{p} \lesssim \normQRP{F}{1}{b}{c}, \quad \forall F \in L^1_t\spaceLbc.
\end{equation*}
The claim is proved for it in the next paragraph. Analogously, we obtain that $b\leq c$. Indeed, the latter estimate is equivalent to
\begin{equation*}
 \normQRP{TT^*F}{\infty}{b'}{c'} \lesssim \normQRP{F}{q'}{r'}{p'},  \quad \forall F \in \spaceLqrpd.
\end{equation*}
The exponent triplet must be KT-acceptable (proved in the next paragraph) and thus $b'\geq c'$. In fact, conditions \eqref{cond: range r} and \eqref{cond: range r b} imply that either $p < b \leq a \leq c < r$ ($p<b$), or $a=b=c=p=r$ and $q=\infty$.

We do not have a suitable counterexample showing the necessity of the upper bound $r^*(c)$ in Theorem \ref{thm: Strich admiss mix KT} for the validity of the generalized homogeneous estimates (in the case when $b\neq c$, $n>1$).

\subsection{Generalized inhomogeneous estimates} \label{sect: counterex inhom}
Let us consider now the inhomogeneous \Str estimate
\begin{equation*}
 \normQRP{W(t)F}{q}{r}{p} \lesssim \normQRP{F}{\tilde q'}{\tilde r'}{\tilde p'}. 
\end{equation*}
By scaling, see Lemma \ref{lem: scaling}, we obtain that the restrictions
\begin{align*}
 \rp q + \rp {\tilde q} = n \left(1 - \rp r - \rp {\tilde r} \right), \quad \HM(p,r)=\HM(\tilde p', \tilde r') \stackrel{def}{=}a,
\end{align*}
are necessary.

Consider $F(t,x,v) = \chi \left( 0 \leq t \leq 1, \abs{x} \leq 1, \abs{v} \leq 1 \right).$ When $t \gg 1$ we have that
\[
 \{TT^*F\}(t) = W(t)F \approx \chi \left( \abs{v- \frac x {t}} \leq \frac 1 t, \abs{v} \leq 1 \right) \approx \chi \left\{ v \sim \frac 1 t, x \sim t\right\}.
\]
Hence,
\[
 \normBC{W(t)F}{r}{p} \sim t^{\frac n r - \frac n p}, \quad t \gg 1.
\]
It follows that $ \normQRP{W(t)F}{q}{r}{p} < \infty$ only if
\begin{equation} \label{cond: inhom accep}
 \left(\frac n r - \frac n p\right)q < -1, \quad \text{or if}\;\, q=\infty, r=p.
\end{equation}
By the duality Lemma \ref{lem: dual}, the dual exponents $\tript$ must also satisfy \eqref{cond: inhom accep}. Thus we have that the conditions $p \leq  r$ and $\tilde p \leq \tilde r$ are necessary for the validity of the considered estimate. The same conclusion applies for the $TT^*$-operator.

We now show that conditions
\begin{align*}
   \frac 1 q + \frac 1 {\tilde q} \leq 1,   \quad \frac 1 r + \frac 1 {\tilde r} \leq 1,
 \end{align*}
are necessary for the validity of the considered estimate. Indeed, the claim follows from the translation invariance of $TT^*$ in $t$ and $x$, H\"ormander's lemma \ref{lem: Horm}, and the equivalence of the considered estimates for $TT^*$ and $W(t)$. Note also that the cases when $\tilde q=1$ or $\tilde r=1$ are trivial and for example by duality can always be replaced by the cases $q=1$ or $r=1$. Thus we have verified that $\trip$ and $\tript$ must be two jointly KT-acceptable exponent triplets, apart from the necessity of condition
\begin{equation*}
        \frac {n -1}{p'} < \frac {n}{\tilde{r}}, \quad \frac {n-1}{\tilde{p}'} < \frac {n}{r}, \quad n>1 
\end{equation*}
for which we do not have a suitable counterexample. However, we can show that the similar condition
\begin{equation*}
        \frac n {p'} < \rp {\tilde q} + \frac {n}{\tilde{r}}, \quad \frac {n}{\tilde{p}'} < \rp {q} + \frac {n}{r}, 
\end{equation*}
is sharp. Indeed, the latter is a direct consequence of \eqref{cond: KT accept} and \eqref{cond: joint KT accept 1}. The latter condition implies the former whenever $p' \leq \tilde q$ and $\tilde p' \leq q$. Thus, if there are some other global inhomogeneous estimates for $W(t)$ not included in Theorem \ref{thm: glob inhom est KT}, they must belong to the range $\tilde q < p'$ or $q < \tilde p'$.

\subsection{Local inhomogeneous estimates} \label{sect: counterex loc inhom}

In this paragraph we show the fact that in the context of the KT equation the local inhomogeneous estimates do not exist in a ``full neighborhood'' around a given local inhomogeneous \Str estimate. This presents an obstruction for the application of the perturbation techniques of Keel and Tao \cite{KT} and their extension by Foschi \cite{F}. The endpoint estimates (of the type that lie on the hypothenuse $AB$ in Figure \ref{fig: 1}) remain unresolved.

For example, consider the estimate
\begin{equation}\label{eq: inhom Strich loc}
 \norm{W(t)F}_{L^{q}_t \left([2,3]; L^{r}_x L^{p}_v\right)} \lesssim \norm{F}_{L^{\tilde q'}_t([1,2]; L^{\tilde r'}_xL^{\tilde p'}_v)}. 
\end{equation}

Take $F(t,x,v)=\chi \left(t \in [0, 1],\; (x,v) \in Q_R \right)$, where by $Q_R$ we denote the square of side length $2R$ centered at the origin of $\R^{2n}$. If we denote $\normInfty{x} = \sup_{1 \leq i \leq n} \abs{x_i}$, for $x$ $=$ $(x_1,..,x_n)$, we can write the latter as
\begin{align*}
 Q_R = \left\{ (x,v) : \normInfty{x} \leq R,\; \normInfty{v} \leq R \right\}.
\end{align*}
Hence,
\[
 \norm{F}_{L^{\tilde q'}_t([1,2]; L^{\tilde r'}_xL^{\tilde p'}_v)} \sim R^{\frac n {\tilde p'} + \frac n {\tilde r'}}.
\]
We now set $\tau=t-s$, and consider the set $Q_R(\tau)$ given by
\begin{align*}
 \normInfty{x - \tau v} \leq R,\; \normInfty{v} \leq R.
\end{align*}
Then, for $t \in [2,3]$, $s \in [0,1]$, equivalently for $\tau \in [1,3]$, we have the inclusions
\[
 Q_{R/4} \subset Q_R(\tau) \subset Q_{4R}.
\]
Hence,
\[
 \norm{W(t)F}_{L^{q}_t \left([2,3]; L^{r}_x L^{p}_v\right)} \sim R^{\frac n p + \frac n r}.
\]
We conclude that condition
\[
 \frac 1 r + \frac 1 p = \frac 1 {\tilde r'} + \frac 1 {\tilde p'}
\]
is necessary for the validity of the local estimates \eqref{eq: inhom Strich loc}.

\section{Remaining unresolved \Str estimates} \label{sect: que}

Here we collect some of the remaining estimates for the KT equation that need to be resolved in order the full range of validity of \Str estimates to be known.

\begin{enumerate}
 \item The endpoint homogeneous estimate in higher dimensions $n>1$
\begin{align*} 
\normQRP{U(t)f}{a}{r^*(a)}{p^*(a)} \lesssim \normLa{f}{a}.
\end{align*}
 \item \label{en: 2} The full range of the non-endpoint inhomogeneous estimates
\begin{equation*}
 \normQRP{W(t)F}{q}{r}{p} \lesssim \normQRP{F}{\tilde q'}{\tilde r'}{\tilde p'}, \quad q>\tilde q'.
\end{equation*}
 In particular, one needs to either show that the condition
\begin{equation*}
        \frac {n -1}{p'} < \frac {n}{\tilde{r}}, \quad \frac {n-1}{\tilde{p}'} < \frac {n}{r},
\end{equation*}
is necessary, or otherwise find and prove the remaining estimates.
 \item The endpoint inhomogeneous estimates with either $q=\tilde q'$ or $q=\infty$ or $\tilde q=\infty$.
 \item The full range of the local inhomogeneous estimates. Equivalently, either show that $\mathcal E_0=\mathcal E$, or otherwise find and prove the remaining estimates.
\end{enumerate}

\nocite{BP}

\bibliographystyle{amsplain}
\bibliography{research_Stat_2}
\end{document}